%% file: main.tex
\documentclass[10pt]{article}
\usepackage[a4paper]{geometry}
\pdfoutput=1

\usepackage[pdfencoding=unicode, hidelinks]{hyperref}
\hypersetup{
	colorlinks,
  breaklinks=true,
	linkcolor={red},
	citecolor={blue},
}
\usepackage{etoolbox}
\cslet{blx@noerroretextools}\empty %
\usepackage{mathtools}
\usepackage{mathrsfs}
\usepackage{amsthm}
\usepackage[nameinlink]{cleveref}
\usepackage{amsfonts}
\usepackage{amssymb}
\usepackage{autonum}
\usepackage{orcidlink}
\usepackage[symbol]{footmisc}

\usepackage[utf8]{inputenc}
\usepackage[shortlabels]{enumitem}
\usepackage[final]{microtype}
\usepackage{bm}
\usepackage{csquotes}
\usepackage[english]{babel}

\theoremstyle{plain}
\newtheorem{teo}{Theorem}[section]
\theoremstyle{plain}
\newtheorem{cor}[teo]{Corollary}
\theoremstyle{plain}
\newtheorem{lemma}[teo]{Lemma}
\theoremstyle{plain}

\theoremstyle{plain}
\newtheorem{prop}[teo]{Proposition}
\theoremstyle{definition}
\newtheorem{remark}[teo]{Remark}
\theoremstyle{definition}
\newtheorem{example}[teo]{Example}

\renewcommand{\r}{\mathbb{R}}

\newcommand{\n}{\mathbb{N}}

\newcommand{\todeb}{\rightharpoonup}

\newcommand{\restr}[2]{\left.\kern-\nulldelimiterspace#1\vphantom{\big|}\right|_{#2}}
\newcommand{\orot}{\mathcal{R}}
\newcommand{\xprimo}{\begin{pmatrix}
		x' \\ 0
\end{pmatrix}}
\newcommand{\umatrix}[1]{\begin{pmatrix}
		#1 \\ 0
\end{pmatrix}}
\newcommand{\lmatrix}[1]{\begin{pmatrix}
		0 \\ #1
\end{pmatrix}}
\newcommand{\rskew}{\r_{\skw}^{3 \times 3}}
\newcommand{\aisolin}{\mathcal{A}_{\iso}^{\lin}}
\newcommand{\aiso}{\mathcal{A}_{\iso}}
\newcommand{\adet}{\mathcal{A}_{\det}}

\DeclareMathOperator{\vspan}{span}

\DeclareMathOperator{\Id}{Id}

\DeclareMathOperator{\SO}{SO}

\DeclareMathOperator{\skw}{skew}
\DeclareMathOperator{\sym}{sym}
\DeclareMathOperator{\dist}{dist}

\DeclareMathOperator{\curl}{curl}

\DeclareMathOperator*{\argmax}{argmax}

\DeclareMathOperator{\tang}{T\orot}
\DeclareMathOperator{\normal}{N\orot}
\DeclareMathOperator{\iso}{iso}
\DeclareMathOperator{\lin}{lin}
\DeclareMathOperator{\VK}{VK}
\DeclareMathOperator{\KL}{K}

\Crefname{lemma}{lemma}{lemmas}
\Crefname{lemma}{Lemma}{Lemmas}
\Crefname{prop}{proposition}{proposition}
\Crefname{prop}{Proposition}{Propositions}
\Crefname{cor}{corollary}{corollaries}
\Crefname{cor}{Corollary}{Corollaries}
\Crefname{remark}{remark}{remarks}
\Crefname{remark}{Remark}{Remarks}
\Crefname{teo}{theorem}{theorems}
\Crefname{teo}{Theorem}{Theorems}
\Crefname{section}{Section}{Sections}
\Crefname{app}{Appendix}{Appendices}
\Crefname{example}{Example}{Examples}

\title{Stability of the Von Kármán regime for thin plates under Neumann boundary conditions}

\author{Edoardo Giovanni Tolotti\footnote{Dipartimento di Matematica F. Casorati, Universit\`a di Pavia, Italy. \\ \indent Email address: edoardogiovann.tolotti01@universitadipavia.it} \hspace{.5mm} \orcidlink{0009-0004-3279-245X}}

\date{}

\begin{document}

\maketitle
\let\thefootnote\relax\footnotetext{This is a post-peer-review, pre-copyedit version of an article published in ESAIM: Control, Optimisation and Calculus of Variations. The final authenticated version is available in Open Access at \url{https://doi.org/10.1051/cocv/2025050}}
\thispagestyle{empty}

\input{abstract}

\input{introduction}

\input{main_results}

\input{stability}

\input{minima}

\appendix

\input{gamma_convergence}
\input{properties_R}
\input{acknowledgments}

\bibliographystyle{plain}
\bibliography{bibliography}

\end{document}

%% file: abstract.tex
\section*{Abstract}

We analyse the stability of the Von K\'arm\'an model for thin plates subject to pure Neumann conditions and to dead loads, with no restriction on their direction. We prove a stability alternative, which extends previous results by Lecumberry and M\"uller in the Dirichlet case. Because of the rotation invariance of the problem, their notions of stability have to be modified and combined with the concept of optimal rotations due to Maor and Mora. Finally, we prove that the Von K\'arm\'an model is not compatible with some specific types of forces. Thus, for such, only the Kirchhoff model applies.

\vspace{.2cm}

\noindent\textbf{Keywords}: Dimension reduction $\cdot$ Thin plates $\cdot$ Nonlinear elasticity $\cdot$ $\Gamma$-convergence

\noindent\textbf{AMS Classification}: 74K20 $\cdot$ 74B20 $\cdot$ 74G65 $\cdot$ 49J45

%% file: introduction.tex
\section{Introduction}

The Von K\'arm\'an model for plates was introduced in the early years of the XX century in the works of F\"oppl and Von K\'arm\'an \cite{FOPPL1907, VONKARMAN1910}. Despite being widely used by engineers, it took almost a century to see a rigorous mathematical derivation, obtained by Friesecke, James and M\"uller in \cite{FRIESECKE2006}, building on their pioneering rigidity estimate \cite{FRIESECKE2002}. In that work, the authors derived the Von K\'arm\'an model computing the $\Gamma$-limit of a suitable rescaling of the three-dimensional nonlinear elastic energy, as the thickness $h$ of the plate vanishes. Then, new mathematical questions naturally arose. Without attempting to be exhaustive, we recall some lines of research: derivation of viscoelastic Von K\'arm\'an models for plates \cite{FRIEDRICH2020}, homogenization of Von K\'arm\'an plates models \cite{VELCIC2016,NEUKAMM2013}, and analysis in the dynamic case of the Von K\'arm\'an equations \cite{ABELS2010,ABELS2011}. Furthermore, one may wonder whether, and how, boundary conditions and applied forces may change the energy scaling and the behaviour of quasi-minimizers.

In this work, we are interested in dead loads of body type. Nevertheless, we point out that the same analysis can be carried out for loads of surface type or for a combination of the two. In this framework, the main difficulty to overcome is the loss of compactness for sequences whose total energy scales like the elastic energy in the Von K\'arm\'an regime. To better understand this issue, let us briefly describe it. In the Von K\'arm\'an setting, the elastic energy $E_h$ per unit volume scales like $h^4$. Since the in-plane displacement scales as $h^2$, it is natural to assume the planar forces to scale like $h^2$ so that the work done by the forces is consistent with the Von K\'arm\'an regime. However, such a choice is also compatible with the Kirchhoff regime (firstly introduced in \cite{KIRCHOFF1850} and rigorously derived in \cite{FRIESECKE2002}) and with the so-called constrained Von K\'arm\'an regime (derived as a $\Gamma$-limit in \cite{FRIESECKE2006}). In these two cases, we have, respectively,
\begin{enumerate}[(i)]
	\item displacement of order $1$ and energy per unit volume of order $h^2$,
	\item displacement of order $h^{\alpha - 2}$ and energy per unit volume of order $h^\alpha$ for some $2 < \alpha < 4$.
\end{enumerate}
Note that in both scenarios the scaling of the work done by the forces is compatible with the corresponding elastic energy regime. Thus, a sequence of deformations $y_h$ with total energy of order $h^4$ may have elastic energy that scales as $h^\alpha$ for any $2 \leq \alpha \leq 4$. In particular, if $\alpha < 4$, then such a sequence has unbounded elastic energy in the Von K\'arm\'an regime, resulting in the aforementioned loss of compactness. This phenomenon can be interpreted as an instability of the Von K\'arm\'an model (see \cite{LECUMBERRY2009}).

The situation is different when the applied forces are purely normal. Indeed, in this case, the natural scaling for forces is $h^3$ (that pairs with the normal displacement of order $h$). As a consequence, there is no ambiguity between the elastic energy regimes. This setting was studied in the original work of Friesecke, James and M\"uller \cite{FRIESECKE2006}. Further analysis in the sole presence of normal forces was carried out in the dynamic setting in \cite{ABELS2010,ABELS2011}.

The more general case with planar forces has been treated by Lecumberry and M\"uller in \cite{LECUMBERRY2009} using a clever exclusion principle. They noted that there is a critical load $f$ that leads to the loss of validity of the Von K\'arm\'an model. Under some additional assumptions, they also proved that beyond this critical load, the infimum of the total Von K\'arm\'an energy is $-\infty$ (see also \cite{MADDALENA2018} for a further analysis of critical points of the Von K\'arm\'an energy). However, to avoid the mix-up of planar and normal components (of both forces and displacement) due to rotation invariance, they had to assume that part of the boundary was subject to a Dirichlet condition.

In the present work we extend this analysis to the purely Neumann case. Since the body is free to rotate, one cannot distinguish between normal and planar components of the applied forces. Thus, we suppose to have a sequence of forces $f_h$ that scale in all directions as $h^2$. For simplicity, we further assume the sequence to be of the form $f_h = h^2 f$ for some given $f$.

The first question to understand is how the load affects the rotation invariance of the plate. In general, one cannot expect the body to prefer just one specific rotation as in the case of clamped boundary conditions. It turns out that the concept of optimal rotations introduced by Maor and Mora in \cite{MAOR2021} is exactly the one needed. The set $\orot$ of such rotations is a submanifold of $\SO(3)$ that in our framework enjoys some additional properties which follow by the two-dimensional nature of the problem.

Secondly, we investigate how the stability conditions defined in \cite{LECUMBERRY2009} can be extended and how they relate to the rotational degree of freedom that the plate enjoys. We prove that one of the following alternatives holds (see \Cref{teo:stability_alternative} for a precise statement):
\begin{itemize}
	\item either the load is strong enough to have a non-trivial minimizer of the Kirchhoff model (failure of the stability condition \ref{item:stability_KL}),
	\item or the load is strong enough to have a non-trivial minimizer of the constrained Von K\'arm\'an model (failure of the stability condition \ref{item:stability_VK}),
	\item or the Von K\'arm\'an model is valid.
\end{itemize}
This result is similar in spirit to \cite[Theorem 4]{LECUMBERRY2009}.
Moreover, in \Cref{teo:implication_stability} we show that the stability condition \ref{item:stability_KL} implies condition \ref{item:stability_VK} as soon as the intensity of the load decreases.
The above implication is analogous to \cite[Theorem 6]{LECUMBERRY2009}.

Compared to the analysis in \cite{LECUMBERRY2009}, we observe a new phenomenon, which is one of the main novelties of this work: if for some optimal rotation $R$ we have $R^T f \cdot e_3 \neq 0$, then the stability condition \ref{item:stability_KL} must fail and both the Von K\'arm\'an model and its constrained version do not apply. More precisely, whenever $R^T f \cdot e_3 \neq 0$, every sequence of quasi-minimizers, whose total energy scales like $h^4$, has unbounded elastic energy in both the Von K\'arm\'an and the constrained Von K\'arm\'an regimes. Note that $e_3$ has a privileged role since it is the direction along which the plate is thin. The precise statement is given in \Cref{teo:stability_false}. One can interpret this result in the following way: it is possible to have a nontrivial minimizer of the Kirchhoff model either increasing the load (as already shown by Lecumberry and M\"uller in \cite{LECUMBERRY2009}) or applying a force whose direction leads to an optimal rotation $R$ of the plate such that $R^T f \cdot e_3 \neq 0$.

Lastly, in a similar fashion to \cite[Theorem 27]{LECUMBERRY2009}, we prove that if \ref{item:stability_VK} holds and $R^T f \cdot e_3 = 0$ for every optimal rotation, the total Von K\'arm\'an energy attains its infimum. Conversely, if \ref{item:stability_VK} fails, the Von K\'arm\'an total energy is unbounded as soon as the load undergoes a slight increase (i.e., $f$ is the critical load). These results are proved in \Cref{teo:JVK_has_minimum}.

The paper is organized as follows. In \Cref{section:notations} we introduce the notation and state the main results.
In \Cref{section:stability_alternative} we prove \Cref{teo:stability_false}--\ref{teo:implication_stability}, while \Cref{teo:JVK_has_minimum} is proved in \Cref{section:minimum}.
Two appendices conclude the manuscript: in \Cref{section:gamma_convergence} we recall known results of $ \Gamma $-convergence of the elastic energy, while in \Cref{app:optimal_rotations} we study fine properties of the set of optimal rotations.

%% file: main_results.tex
\section{Notations and main results}\label{section:notations}

\subsection{Notations and functional setting}

We denote by $W \colon \r^{3 \times 3} \to [0, +\infty]$ the elastic energy density. We assume $W$ to be Borel measurable and to satisfy the following standard hypotheses:
\begin{enumerate}[start=1,label={(W\arabic*)}]
	\item $W(A) = 0 \iff A \in \SO(3)$, \label{item:minimization_W}
	\item $W$ is $C^2$ in a neighbourhood of $\SO(3)$, \label{item:regularity_W}
	\item $W$ is frame indifferent, i.e., $W(RA) = W(A)$ for every $R \in \SO(3)$ and for every $A \in \r^{3 \times 3}$, \label{item:frame_indifference}
	\item $W(A) \geq C\dist^2(A, \SO(3))$ for every $A \in \r^{3 \times 3}$, for some $C > 0$. \label{item:lower_bound_W}
\end{enumerate}
We will denote by $Q: \r^{3 \times 3} \to \r$ the quadratic form $D^2W(\Id)$ and by $\bar Q: \r^{2 \times 2} \to \r$ the reduced quadratic form
\[
	\bar Q(A) = \min_{a \in \r^3} Q(A + a \otimes e_3 + e_3 \otimes a) \,.
\]
By \ref{item:lower_bound_W} both $Q$ and $\bar Q$ are coercive over the set of symmetric matrices. 

We consider a thin plate $\Omega_h = S \times (-\frac{h}{2}, \frac{h}{2}) = S \times hI$ where $S \subset \r^2$ is an open, simply connected and bounded set and $I = (-\frac{1}{2}, \frac{1}{2})$. In terms of regularity of $\partial S$, we assume the following condition:
\begin{equation}\label{eq:boundary_condition}
	\begin{gathered}
		\text{there is a closed subset } \Sigma \subset \partial S \text{ with } \mathcal{H}^1(\Sigma) = 0 \text{ such that}\\
		 \text{the outer unit normal } \vec n \text{ to } S \text{ exists and is continuous on } \partial S \backslash \Sigma\,.
	\end{gathered}
\end{equation}
This property is called \textit{condition $(\ast)$} in \cite{HORNUNG2011}. We write $\Omega$ for the rescaled plate, that is, $\Omega = \Omega_1$. The symbol $\nabla_h y$ denotes the rescaled gradient of $y$, and it is defined as follows:
\[
	(\nabla_h y)_{ij} = \begin{cases}
		\partial_j y_i & \text{if } j \neq 3\,,\\
		\frac{1}{h} \partial_3 y_i & \text{otherwise}\,.
	\end{cases}
\]
The elastic energy is written in terms of $\Omega$ as follows
\[
	E_h(y) = \int_\Omega W(\nabla_h y) \, dx \quad \text{for} \, \, y \in W^{1, 2}(\Omega; \r^3) \,.
\]
Given a matrix $A \in \r^{3 \times 3}$ we will write $A'$ to denote the $2 \times 2$ submatrix obtained by removing the third row and column. Similarly, for a vector $ v \in \mathbb{R}^3 $ we will often write $v'$ in place of $(v_1, v_2)$ and $\nabla'$ instead of $\nabla_{x_1, x_2}$. Whenever we will sum or multiply matrices and vectors with different dimensions we will imply that the smaller one is naturally embedded in the bigger space by adding zeros in the missing entries. For example, if $A \in \r^{2 \times 2}$ and $G \in \r^{3 \times 3}$ the expression $A + G$ means $\iota(A) + G$ where
\[
\iota: \r^{2 \times 2} \hookrightarrow \r^{3 \times 3} \,, \quad  F \mapsto \begin{pmatrix}
A & 0 \\
0 & 0
\end{pmatrix} \,.
\]

Given two sequences $ a_{h} $ and $ b_{h} $ we will write $ O(a_{h}, b_{h}) $ meaning that there exists a constant $ C > 0 $ independent of $ h $ such that $ |O(a_{h}, b_{h})| \leq C(|a_{h}| + |b_{h}|) $.
In particular, if $ |a_{h}| \ll |b_{h}| $ for $ h \to 0 $, we have $ O(a_{h}, b_{h}) = O(b_{h}) $.

We will assume the applied forces to be of the form
\begin{equation}\label{eq:forces_form}
	f_h = h^2 f \,,
\end{equation}
with $f \in L^2(S; \r^3)$, $f$ not identically equal to $0$.

Following \cite{MAOR2021}, we introduce the set $\orot$ of optimal rotations, defined as 
\[
\orot = \argmax_{R \in \SO(3)} F(R) \,,
\]
where
\[
F(A) = \int_S f \cdot A \xprimo \, dx' \,.
\]
The set $\orot$ is a closed, connected, boundaryless and totally geodesic submanifold of $\SO(3)$ \cite[Proposition 4.1]{MAOR2021}. We refer to \Cref{app:optimal_rotations} for further properties of $\orot$. 

The total energy for a deformation $y \in W^{1, 2}(\Omega; \r^3)$ can be written as
\begin{equation}\label{eq:definition_work}
J_h(y) = E_h(y) - \int_\Omega f_h \cdot y \, dx = E_h(y) - h^2\int_\Omega f \cdot y \, dx\,.
\end{equation}
We suppose that
\begin{equation}\label{eq:mean_condition_forces}
\int_S f \, dx' = 0
\end{equation}
to avoid the trivial case in which the total energy has no lower bound. For a pair $(u, v) \in W^{1,2}(S; \r^2) \times W^{2, 2}(S)$ we define the Von K\'arm\'an energy
\[
E^{\VK}(u, v) = \dfrac{1}{8} \int_S \bar Q(\nabla' u^T + \nabla' u + \nabla' v \otimes \nabla' v) \, dx' + \dfrac{1}{24}\int_S \bar Q ((\nabla')^2 v)\, dx'\,.
\]
We will often consider $E^{\VK}$ restricted to the set of geometrically linearized isometries, namely
\[
\mathcal{A}_{\iso}^{\lin} = \left\{(u, v) \in W^{1,2}(S; \r^2) \times W^{2, 2}(S) \colon \nabla' u^T + \nabla' u + \nabla' v \otimes \nabla' v = 0 \text{ a.e. in } S \right\}\,.
\]
On this set $E^{\VK}$ only depends on $v$, and we have
\[
E^{\VK}(u, v) = \dfrac{1}{24}\int_S \bar Q((\nabla')^2 v)\, dx' \quad \forall \, (u, v) \in \aisolin \,.
\]
We define the set of isometric embeddings of $S$ as
\[
\aiso = \left\{y \in W^{2, 2}(S; \r^3) \colon \nabla' y^T \nabla' y = \Id \text{ a.e. in } S\right\} \,.
\]
For $y \in \aiso$ we introduce the Kirchhoff energy,
\begin{equation}
E^{\KL}(y) = \dfrac{1}{24} \int_S \bar Q(\nabla' y^T \nabla' \nu) \, dx' \,,
\end{equation}
where $\nu = \partial_1 y \land \partial_2 y$. Finally, we define the total energy in the Von K\'arm\'an and Kirchhoff regimes, respectively, as
\begin{align}
J^{\VK}(u, v, R, W) & = E^{\VK}(u, v) - \int_S f \cdot R \umatrix{u} \, dx' - \int_S f \cdot RW \lmatrix{v} \, dx' \\
& \hphantom{=} \, \, - \int_S f \cdot RW^2 \umatrix{x'}\, dx' \,, \\
J^{\KL}(y) & = E^{\KL}(y) - \int_S f \cdot y \, dx' \,.
\end{align}
The first functional is defined for every $(u, v) \in W^{1,2}(S; \r^2) \times W^{2, 2}(S)$, $R \in \orot$ and $W \in \normal_R$ (see \eqref{eq:definition_normal} in \Cref{app:optimal_rotations} for the definition of $\normal_R$). A quadruplet $(u, v, R, W)$ as above will be called admissible. The Kirchhoff functional is defined for every $y \in \aiso$.
These energies can be interpreted as the $ \Gamma $-limit of the corresponding rescalings of $ J_{h} $.
However, the $ \Gamma $-limit result alone is not satisfactory, since we lack the corresponding compactness properties for sequences with bounded total rescaled energy.

\subsection{Main results}

Similarly to the Dirichlet case treated in \cite{LECUMBERRY2009}, an exclusion principle involving the stability of $J^{\VK}$ and $J^{\KL}$ can be used to study the limit of minimizing sequences in the Von K\'arm\'an regime. In our setting, these stability conditions read as follows:
\begin{enumerate}[(S1)]
	\item $J^ {\KL}(y) \geq 0$ for every $y \in \aiso$ and, if $J^{\KL}(y) = 0$, then $y = \hat R \umatrix{x'}$ for some $\hat R \in \SO(3)$, \label{item:stability_KL}
	\item $J^{\VK}(u,v,R,W) \geq 0$ for every admissible quadruplet $(u,v, R, W)$ with $(u, v) \in \aisolin$ and, if $J^{\VK}(u, v, R, W) = 0$ for some $(u, v) \in \aisolin$, then $v$ is affine.\label{item:stability_VK}
\end{enumerate}
Conditions \ref{item:stability_KL}--\ref{item:stability_VK} have to be interpreted as follows: whenever a deformation minimizes the (non-negative) total energy then it must be a deformation with zero elastic energy.
In our framework, if \ref{item:stability_KL} holds, then the following compatibility condition is in force:
\begin{equation}\label{eq:compatibility}
    R^{T} f \cdot e_{3} = 0, \qquad \forall \, R \in \orot. \tag{C}
\end{equation}
This is the main statement of \Cref{teo:stability_false}.
Compatibility condition \eqref{eq:compatibility} is the rotation invariant generalization of the usual assumption on the scaling of the normal component of the forces, see for example \cite{FRIESECKE2006}. Indeed, the standard requirement $f_h \cdot e_3 = O(h^3)$ in our setting translates to $f \cdot e_3 = 0$ (see \eqref{eq:forces_form}).

From now on, unless otherwise stated, $ (y_{h}) \subset W^{1,2}(S;\mathbb{R}^3) $ will denote a quasi-minimizing sequence for $ \frac{1}{h^{4}}J_{h} $, namely
\begin{equation}\label{eq:quasi-minimizing}
    \limsup_{h \to 0} \frac{1}{h^4}\left(J_h(y_h) - \inf_y J_h(y)\right) = 0 \,.
\end{equation}

\begin{teo} \label{teo:stability_false}
    Assume that \eqref{eq:compatibility} is not valid.
    Then \ref{item:stability_KL} fails.
    Moreover, up to a subsequence, every sequence $ (y_{h}) $ of quasi-minimizers in the sense of \eqref{eq:quasi-minimizing} converge strongly in $ W^{1, 2}(\Omega;\mathbb{R}^3) $ to a minimizer $ \bar y \in \aiso $ of $ J^{\KL} $ with $ \bar y \neq R \xprimo $ for every $ R \in \SO(3) $.
\end{teo} 

\Cref{teo:stability_false} shows that in the purely Neumann case, some forces are incompatible with the Von K\'arm\'an regime. 
In particular, if \eqref{eq:compatibility} is not in force, the energy of any sequence of quasi-minimizers as in \eqref{eq:quasi-minimizing} scales like $ h^{2} $, namely
\begin{equation}
    0 < \liminf_{h \to 0} \frac{1}{h^{2}}E_{h}(y_{h}) < + \infty.
\end{equation}

Next, we state the stability alternative analogue to \cite[Theorem 4]{LECUMBERRY2009}.

\begin{teo}\label{teo:stability_alternative}
	Let $(y_h) \subset W^{1, 2}(S; \r^3)$ be a sequence of quasi-minimizers in the sense of \eqref{eq:quasi-minimizing}. Suppose that conditions \ref{item:stability_KL}--\ref{item:stability_VK} hold true. Then $\limsup_{h \to 0} \frac{1}{h^4} E_h(y_h) \leq C$ and there are sequences $(R_h) \subset \SO(3)$ and $(c_h) \subset \r^3$ such that, setting $\tilde y_h$ as $\tilde y_h = R_h^T y_h + c_h$, we have the following convergences (up to a subsequence):
	\begin{enumerate}[(a)]
		\item $u_h = \dfrac{1}{h^2}\displaystyle\int_I \left(\tilde y_h' - x'\right) \, dx_3 \todeb \bar u$ in $W^{1, 2}(S; \r^2)$,
		\item $v_h = \dfrac{1}{h}\displaystyle\int_I \tilde y_{h, 3} \, dx_3 \to \bar v$ in $W^{1, 2}(S)$ with $\bar v \in W^{2, 2}(S)$,
		\item $R_h \to \bar R \in \orot$,
		\item $\frac{1}{h}(P(R_h) - R_h) \to \bar R \bar  W $ with $\bar W \in \normal_R$,
	\end{enumerate}
	where $P: \SO(3) \to \orot$ is the projection onto $\orot$. Finally, the quadruplet $(\bar u, \bar v, \bar R, \bar W)$ minimizes $J^{\VK}$.
\end{teo}

Similarly to \cite[Theorem 6]{LECUMBERRY2009} we can show that \ref{item:stability_KL} and \ref{item:stability_VK} are in a relationship, with the former being essentially stronger than the latter.

\begin{teo}\label{teo:implication_stability}
	Suppose that \ref{item:stability_KL} holds. Then $J^{\VK}(u, v, R, W) \geq 0$ for every admissible quadruplet with $(u, v) \in \aisolin$. Moreover, \ref{item:stability_VK} holds for the functional
	\begin{align}
	J^{\VK}_\varepsilon(u, v, R, W) & = E^{\VK}(u, v) - (1-\varepsilon)\int_S f \cdot R \umatrix{u} \, dx' \\
	& \hphantom{=} \, \, - (1-\varepsilon)\int_S f \cdot RW \lmatrix{v} \, dx' - (1-\varepsilon)\int_S f \cdot RW^2 \umatrix{x'}\, dx' \\
	\end{align}
	for every $\varepsilon \in (0, 1)$.
\end{teo}

In general the previous result does not hold when $ \varepsilon = 0 $. Indeed, one can only deduce the positivity of $ J^{\VK} $ but not the triviality of the minimizers.

The stability conditions are deeply linked to the attainment of the infimum of $J^{\VK}$. Indeed, we will prove the following.

\begin{teo} \label{teo:JVK_has_minimum}
	Suppose that the stability condition \ref{item:stability_VK} and the compatibility condition \eqref{eq:compatibility} hold true, and that $\dim \orot = 1$. Then $J^{\VK}$ attains its minimum over all admissible quadruplets $(u, v, R, W)$. Instead, if \ref{item:stability_VK} fails, then for every $\varepsilon > 0$ the infimum of the functional
	\begin{align}
		J^{\VK}_\varepsilon(u, v, R, W) & = E^{\VK}(u, v) - (1+\varepsilon)\int_S f \cdot R \umatrix{u} \, dx' \\
		& \hphantom{=} \, \, - (1+\varepsilon)\int_S f \cdot RW \lmatrix{v} \, dx' - (1+\varepsilon)\int_S f \cdot RW^2 \umatrix{x'}\, dx' \\
	\end{align}
	is $- \infty$.
\end{teo}

As for \Cref{teo:implication_stability}, also \Cref{teo:JVK_has_minimum} might not hold for $ \varepsilon = 0 $ (see also \Cref{remark:eps0}). Roughly speaking, this would mean that the load $ f $ is critical, i.e., as soon as the load increases the Von K\'arm\'an model ceases to be valid.

\begin{remark}
	In \Cref{lemma:dimension_R} and \Cref{remark:dimension_R} it is proved that the dimension of $\orot$ is either zero or one. \Cref{teo:JVK_has_minimum} holds also in the case $\dim \orot = 0$. However, if $\orot$ is a singleton, our setting reduces to the one in \cite{LECUMBERRY2009}. For this reason, we will only give a sketch of the proof for the case $\dim \orot = 0$ (see \Cref{remark:dim_R_0}).
\end{remark}

To prove \Cref{teo:JVK_has_minimum} a careful analysis of the invariance of $J^{\VK}$ along affine perturbations will be needed.

%% file: stability.tex
\section{Stability alternative}\label{section:stability_alternative}

The aim of this section is to prove \Cref{teo:stability_false}--\ref{teo:implication_stability}.

In our arguments, we will often compare the quasi-minimizing sequence $ y_{h} $ (in the sense of \eqref{eq:quasi-minimizing}) with carefully chosen test deformations $ \hat y_{h} $.
Indeed, we have
\begin{equation} \label{eq:small_o_minimizing}
	J_h(y_h) - J_h(\hat y_h) = \inf_y J_h(y) - J_h(\hat y_h) + o(h^4) = o(h^4) \,.
\end{equation}
Passing to the limit in \eqref{eq:small_o_minimizing} we will deduce relevant properties of the quasi-minimizing sequence $ y_{h} $.
The test deformations $ \hat y_{h} $ we will use are inspired by the recovery sequence construction of \cite{FRIESECKE2006}.
For this reason, the interested reader can find the explicit computation of their elastic energy in \Cref{section:gamma_convergence} (\Cref{prop:fine_test_deformation,prop:fine_test_deformation_det}).

In order to prove \Cref{teo:stability_false} it is crucial to have at our disposal the following result, relating the energy scaling of $ y_{h} $ and the compatibility condition \eqref{eq:compatibility}.
Here, and in the rest of the section, we will denote by $(D_h) \subset \r^+$ an infinitesimal sequence.

\begin{teo}\label{teo:compatibility_forces}
	Suppose that $\limsup_{h \to 0} \frac{1}{D_h} E_h(y_h) \leq C$ with $\frac{D_h}{h^2} \to 0$.
    Then \eqref{eq:compatibility} is in force, i.e., $R^T f \cdot e_3 = 0$ for every $R \in \orot$.
\end{teo}

\noindent For a quasi-minimizing sequence as in \eqref{eq:quasi-minimizing} we have that $ E_{h}(y_{h}) \leq Ch^{2} $ (see the proof of \Cref{teo:stability_false}).
Thus, \Cref{teo:compatibility_forces} equivalently ensures that the elastic energy of $ y_{h} $ scales like $ h^{2} $ when \eqref{eq:compatibility} does not hold true.

To prove \Cref{teo:compatibility_forces} we will proceed by steps, one for each possible limit of $ \frac{D_{h}}{h^{4}} $.
Every case corresponds to an elastic energy regime. 
In each step we will compare the quasi-minimizing sequence $ y_{h} $ with test deformations having the same elastic energy scaling.
In order to conclude when $ \frac{D_{h}}{h^{4}} \to \infty $ we will need some auxiliary results regarding linearized isometries (in the language of \cite{FRIESECKE2006}) that we state here, postponing their proof. 

\begin{lemma}\label{lemma:density_adet}
	Let
	\begin{equation}\label{eq:definition_adet}
		\adet= \{v \in W^{2,2}(S) \colon \det((\nabla')^2 v) = 0 \text{ a.e. in } S\} \,.	
	\end{equation}
	Then $\vspan \adet$ is dense in $L^2(S)$.
\end{lemma}

\begin{lemma}\label{lemma:contruction_isometries}
	Let $v \in W^{2, \infty}(S) \cap \adet$. Then there is $u \in W^{2, \infty}(S; \r^2)$ such that the map
	\[
		y(x') = \xprimo + \begin{pmatrix}
			u(x') \\
			v(x')
		\end{pmatrix}
	\]
	is an isometric embedding, i.e., $\nabla' y^T \nabla' y = \Id$ almost everywhere. Moreover, if $\|\nabla' v\|_{\infty} \leq \frac{1}{2}$, we have
	\begin{equation}\label{eq:regularity_u_isometric_embedding}
		\|u\|_{W^{2, \infty}} \leq C(\|(\nabla')^2 v\|_{\infty}\|\nabla' v\|_{\infty} + \|\nabla' v\|_{\infty}^2) \,.
	\end{equation}
\end{lemma}

The first part of the section is thus devoted to the proof of \Cref{teo:compatibility_forces}.
Once \Cref{teo:compatibility_forces} is established, we will move to the proof of \Cref{teo:stability_false,teo:stability_alternative,teo:implication_stability}.

We start by proving that, if we are in the Von K\'arm\'an energy scaling, the quasi-minimizing sequence $ y_{h} $ will converge, up to a subsequence, to a rigid motion given by an optimal rotation. 

\begin{lemma} \label{prop:limit_optimal_rotation}
	Suppose that $E_h(y_h) \leq CD_h$ and $\frac{D_h}{h^2} \to 0$. Let $(R_h) \subset \SO(3)$ be the sequence provided by \Cref{prop:compactness_alpha_greater_2}. Then, up to a subsequence, $R_h \to \bar R \in \orot$.
\end{lemma}
\begin{proof}
	Up to a subsequence, $R_h \to \bar R$ for some $\bar R \in \SO(3)$. To prove that $\bar R \in \orot$ pick a rotation $R \in \SO(3)$ and consider the test deformation
	\begin{equation}\label{eq:test_deformation}
		\hat y_h(x', x_3) = R\begin{pmatrix}
			x' \\ hx_3
		\end{pmatrix}\,.
	\end{equation}
	The elastic energy of $\hat y_h$ is zero, so using the notation of \Cref{prop:compactness_alpha_greater_2} we have
	\begin{align}
		J_h(y_h) - J_h(\hat y_h) & \geq \int_\Omega f_h \cdot \hat y_h \, dx - \int_\Omega f_h \cdot y_h \, dx\\
		& = - h^2 \int_S f \cdot R_h \begin{pmatrix}
			\max\left\{\sqrt{D_h}, \dfrac{D_h}{h^2}\right\} u_h \\ h^{-1} \sqrt{D_h} v_h
		\end{pmatrix} \, dx' \\
		& \hphantom{=} \, \, + h^2 \int_S f \cdot (R - R_h) \xprimo \, dx' \,.
	\end{align}
    Here, we have used that $ f $ does not depend on $ x_{3} $ and the symmetry of $ (-\frac{1}{2}, \frac{1}{2}) $ to deduce that
    \begin{equation}
        \int_{\Omega} f \cdot R \begin{pmatrix}
            0 \\ x_{3}
        \end{pmatrix} \, dx = \int_{\Omega} f \cdot R_{h} \begin{pmatrix}
            0 \\ x_{3}
        \end{pmatrix} \, dx = 0.
    \end{equation}
	Dividing everything by $h^2$ and passing to the limit we deduce by \eqref{eq:small_o_minimizing} and \Cref{prop:compactness_alpha_greater_2}
	\[
		0 \geq  \int_S f \cdot (R - \bar R) \xprimo \, dx' \,,
	\]
	which gives $\bar R \in \orot$.
\end{proof}

\begin{proof}[Proof of \Cref{teo:compatibility_forces}]
    The proof is divided in three steps, one for each possible elastic energy scaling.

	\textbf{Step 1.} We start by considering the case where $D_hh^{-4} \to 0$.
    Let $ R \in \orot $ and consider the test deformation
    \begin{equation}\label{eq:fine_test_deformation}
		\hat y_h(x', x_3) = R \begin{pmatrix}
			x' \\ hx_3
		\end{pmatrix} + R \begin{pmatrix}
			- h^2x_3 \nabla v^T \\ hv
		\end{pmatrix} \,.
	\end{equation}
    By \Cref{prop:fine_test_deformation} we have that $ E_{h}(\hat y_{h}) = O(h^{4}) $.
    Comparing the quasi-minimizing sequence with the test deformations and using that $R \in \orot$ we get
	\begin{equation}
		\begin{aligned}
			J_h(y_h) - J_h(\hat y_h) & \geq O(h^4) + \int_\Omega f_h \cdot \hat y_h \, dx - \int_\Omega f_h \cdot y_h \, dx\\
			& = O(h^4) - \int_S f_h \cdot R_h \begin{pmatrix}
			\sqrt{D_h} u_h \\ h^{-1} \sqrt{D_h} v_h
			\end{pmatrix} \, dx' \\
			& \hphantom{=} \, \, + \int_S f_h \cdot (R - R_h) \xprimo \, dx' +\int_S f_h \cdot R \begin{pmatrix}
				0 \\ hv
			\end{pmatrix} \, dx' \\
			& \geq O(h^4) + h^3 \int_S f \cdot R \begin{pmatrix}
			0 \\ v
			\end{pmatrix} \, dx' \\
			& \hphantom{=} \, \, - h\sqrt{D_h} \int_S f \cdot R_h \begin{pmatrix}
			0 \\ v_h
			\end{pmatrix} \, dx' \,,
		\end{aligned}
		\label{eq:argument_comparison}
	\end{equation}
	where $v_h, u_h, R_h$ are the ones given by \Cref{prop:compactness_alpha_greater_2}.
    Dividing by $h^3$ and passing to the limit, by \eqref{eq:small_o_minimizing} and the fact that $ D_{h}h^{-4} \to 0 $ we deduce that
	\begin{equation}
		0 \geq \int_S f \cdot R \begin{pmatrix}
		0 \\ v
		\end{pmatrix} \, dx' \quad \forall \, R \in \orot\,, \, \, \forall \, v \in C^{\infty}(\bar S) \,,
	\end{equation}
	and by density the same holds for every $ v \in L^{2}(S) $. 
	Since the map
	\[
		v \in L^2(S) \mapsto \int_S R^T f \cdot \begin{pmatrix}
			0 \\ v
		\end{pmatrix} \, dx'
	\]
	is linear it must be identically zero, that is, $R^T f \cdot e_3 = 0$ for any $R \in \orot$. 
	
	\textbf{Step 2.} We move now to the case where $D_hh^{-4} \to D > 0$. Let $R \in \orot$ and $v \in C^\infty(\bar S)$ and consider again the test deformation $\hat{y}_h$ as in \eqref{eq:fine_test_deformation}. Arguing as in \eqref{eq:argument_comparison}, we deduce that
	\begin{align}
		J_h(y_h) - J_h(\hat y_h) \geq h^3 \int_S f \cdot R \begin{pmatrix}
		0 \\ v
		\end{pmatrix} \, dx' - h\sqrt{D_h} \int_S f \cdot R_h \begin{pmatrix}
		0 \\ v_h
		\end{pmatrix} \, dx' + O(h^4) \,,
	\end{align}
	where $v_h, u_h, R_h$ are the ones given by \Cref{prop:compactness_alpha_greater_2}. Let $\bar v$ be the limit of $v_h$ and $\bar R$ the limit of $R_h$. Dividing by $h^3$ and passing to the limit we deduce by \eqref{eq:small_o_minimizing} that
	\begin{equation}
	0 \geq \int_S f \cdot R \begin{pmatrix}
	0 \\ v
	\end{pmatrix} \, dx' - \sqrt{D}\int_S f \cdot \bar R \begin{pmatrix}
	0 \\ \bar v
	\end{pmatrix} \, dx' \quad \forall \, R \in \orot\,, \, \, \forall \, v \in C^{\infty}(\bar S) \,,
	\end{equation}
	and by density the same holds for every $ v \in L^{2}(S) $.
    Arguing as before, we conclude by linearity.
	
	\textbf{Step 3.} Finally, we discuss the case $D_hh^{-4} \to +\infty$.
    Let $v \in C^\infty(\bar S)$ such that $\det((\nabla')^2 v) = 0$ in $S$.
    By \Cref{lemma:contruction_isometries} there exists $\tilde u_h \in W^{2, \infty}(S; \r^2)$ such that
	\[
		\tilde y_h(x') = \xprimo + \begin{pmatrix}
			h^{-2}D_h \tilde u_h \\
			h^{-1}\sqrt{D_h}v
		\end{pmatrix} \,
	\]
	is an isometric embedding, i.e., $ \nabla' \tilde y_{h}^{T} \nabla' \tilde y_{h} = \Id $.
    Note that by \eqref{eq:regularity_u_isometric_embedding} and the fact that $ h^{-1}\sqrt{D_{h}} \to 0 $, we have $ \|\tilde u_{h}\|_{W^{2, \infty}} \leq C $. 
    Let $ R \in \orot $ and define 
	\begin{equation}\label{eq:fine_test_deformation_det}
		\hat y_h = R\tilde y_h + h x_3 R\nu_h \,,
	\end{equation}
	where $\nu_h  = \partial_1 \tilde y_h \land \partial_2 \tilde y_h$.
    By \Cref{prop:fine_test_deformation_det} we have $ E_{h}(\hat y_{h}) = O(D_{h}) $.
    Comparing the test deformation $\hat y_h$ with the minimizing sequence, using that $R \in \orot$, and that \eqref{eq:regularity_u_isometric_embedding} holds true, we get
	\begin{align}
		J_h(y_h) - J_h(\hat y_h) & \geq \int_\Omega f_h \cdot \hat y_h \, dx - \int_\Omega f_h \cdot y_h \, dx + O(D_h)\\
		& = - \int_S f_h \cdot R_h \begin{pmatrix}
		h^{- 2} D_h u_h \\ h^{-1} \sqrt{D_h} v_h
		\end{pmatrix} \, dx' + \int_S f_h \cdot (R - R_h) \xprimo \, dx' \\
		& \hphantom{\geq} \, \, +\int_S f_h \cdot R \begin{pmatrix}
		h^{- 2} D_h \tilde u_h \\ h^{-1} \sqrt{D_h} v
		\end{pmatrix} \, dx' + O(D_h) \\
		& \geq  h\sqrt{D_h} \int_S f \cdot R \begin{pmatrix}
		0 \\ v
		\end{pmatrix} \, dx' - h\sqrt{D_h} \int_S f \cdot R_h \begin{pmatrix}
		0 \\ v_h
		\end{pmatrix} \, dx' \\
		& \hphantom{=} \, \, + O(D_h)\,,
	\end{align}
	where $v_h, u_h, R_h$ are the ones given by \Cref{prop:compactness_alpha_greater_2}. Let $\bar v$ and $\bar R$ be the limits of $v_h$ and $R_h$, respectively. Dividing by $h\sqrt{D_h}$ and passing to the limit we obtain that for every $R \in \orot$ and for every $v \in C^\infty(\bar S)$ such that $\det((\nabla')^2 v) = 0$ we have
	\begin{equation}\label{eq:maximum_R_v_det}
		0 \geq \int_S f \cdot R \begin{pmatrix}
		0 \\ v
		\end{pmatrix} \, dx' - \int_S f \cdot \bar R \begin{pmatrix}
		0 \\ \bar v
		\end{pmatrix} \, dx' \,.
	\end{equation}
	Since $S$ satisfies condition \eqref{eq:boundary_condition}, the density result given by \cite[Theorem 1]{HORNUNG2011} ensures that \eqref{eq:maximum_R_v_det} actually holds for any $v \in \adet$ (see \eqref{eq:definition_adet} for the definition of $\adet$).
    By \Cref{prop:limit_optimal_rotation} $\bar R \in \orot$, hence, choosing $R = \bar R$ we have once again that $\bar v$ maximizes the linear map
	\[
	v \mapsto \int_S \bar R^T f \cdot \lmatrix{v} \, dx'
	\]
	on $\adet$. We note that $\adet$ is not a linear space. However, if $v \in \adet$, then $\lambda v \in \adet$ for any $\lambda \in \r$. Therefore, we conclude that
	\[
		\int_S \bar R^T f \cdot \lmatrix{v} \, dx' = 0 \quad \forall \, v \in \adet \,.
	\]
    Going back to \eqref{eq:maximum_R_v_det}, we deduce that
	\[
		0 \geq \int_S f \cdot R \begin{pmatrix}
		0 \\ v
		\end{pmatrix} \, dx' \quad \forall \, R \in \orot \,, \, \, \forall \, v \in \adet\,.
	\]
	Hence, by linearity the same holds for every $ v \in \overline{\vspan \adet} $, where the closure is taken in $ L^{2}(S) $.
    The conclusion follows from \Cref{lemma:density_adet}.
\end{proof}

We prove now the auxiliary results that we have used. 
\Cref{lemma:density_adet} follows from the well-known universal approximation theorem of single-layer neural networks. For the convenience of the reader, we will recall here the main ideas needed in our setting. We borrow the notation from \cite{CYBENKO1989}, but a similar result can also be found in \cite{CARROLL1989}.

\begin{proof}[Proof of \Cref{lemma:density_adet}]
	\textbf{Step 1}. Let $\sigma(x) = \dfrac{1}{1+e^{-x}}$. We will show that $\sigma$ is discriminatory, i.e., the only signed bounded regular Borel measure $\mu$ on $\bar S$ such that 
	\begin{equation}\label{eq:discriminatory}
	\int_{\bar S} \sigma(y^T x + \theta) \, d\mu(x) = 0 \quad \forall y \in \r^2, \, \, \forall \theta \in \r
	\end{equation}
	is $\mu = 0$. 
	
	Let $\mu$ be such that \eqref{eq:discriminatory} holds. We argue as in \cite[Lemma 1]{CYBENKO1989}. Let $y \in \r^2$, $\lambda, \theta, k \in \r$, and define
	\[
	\sigma_\lambda(x) = \sigma(\lambda(y^T x + \theta) + k) \,.
	\]
	Let
	\[
	\phi(x) = \begin{cases}
	1 \quad & \text{if } y^Tx + \theta > 0 \,,\\
	0 \quad & \text{if } y^Tx + \theta < 0 \,,\\
	\sigma(k) \quad & \text{if } y^Tx + \theta = 0 \,.
	\end{cases}
	\]
	Clearly, $\sigma_\lambda \to \phi$ pointwise as $\lambda \to +\infty$. Moreover, $\|\sigma_\lambda\|_{C^0} \leq 1$ uniformly. Hence, by \eqref{eq:discriminatory} and dominated convergence
	\begin{equation}\label{eq:condition_discrimination}
	0 = \int_{\bar S} \sigma_\lambda \, d\mu \to \int_{\bar S} \phi \, d\mu = \sigma(k)\mu(\Pi_{y, \theta}) + \mu(H_{y, \theta}) \quad \forall \, y \in \r^2 \, \, \forall \theta, k \in \r \,,
	\end{equation}
	where
	\begin{align}
	\Pi_{y, \theta} & = \{x \in S \colon y^T x + \theta= 0\}\,,\\
	H_{y, \theta} & = \{x \in S \colon y^T x+ \theta > 0\}\,.
	\end{align}
	Passing to the limit as $k \to +\infty$ in \eqref{eq:condition_discrimination}, we deduce that
	\begin{equation}\label{eq:condition_discrimination2}
	\mu(\Pi_{y, \theta}) + \mu(H_{y, \theta}) = 0 \quad \forall \, y \in \r^2 \, \, \forall \, \theta \in \r \,.
	\end{equation}
	Similarly, letting $k \to -\infty$ we get
	\begin{equation}\label{eq:condition_discrimination3}
	\mu(H_{y, \theta}) = 0 \quad \forall \, y \in \r^2 \, \, \forall \, \theta \in \r \,.
	\end{equation}
	Fix $y \in \r^2$ and define
	\[
	F_y(h) = \int_{\bar S} h(y^Tx) \, d\mu(x) \,,
	\]
	for every bounded Borel function $h: \r \to \r$. Let $\theta \in \r$. Then,
	\[
	F_y(\chi_{[-\theta, +\infty)}) = \mu(\Pi_{y,\theta}) + \mu(H_{y, \theta}) = 0 \,,
	\]
	where $\chi_{[-\theta, +\infty)}$ is the indicator function of ${[-\theta, +\infty)}$. Similarly, $F_y(\chi_{(-\theta, +\infty)}) = 0$. By the linearity of $F_y$, we deduce that $F_y$ is zero on the indicator function of every interval. By approximation, $F_y(h) = 0$ for every continuous and bounded function $h: \r \to \r$ and for every $y \in \r^2$. In particular,
	\begin{equation}
		\begin{aligned}
			\hat \mu(\xi) & = \int_{\bar S} e^{-i\xi^T x} \, d\mu(x) = \int_{\bar S} (\cos(\xi^T x) + i \sin(\xi^T x)) \, d\mu(x) \\
			& = F_\xi(\cos(x)) + iF_\xi(\sin(x)) = 0 \,, 
		\end{aligned}
	\end{equation}
	where $\hat \mu$ is the Fourier transform of $\mu$. Since $\hat \mu= 0$, it follows that $\mu = 0$.
	
	\textbf{Step 2}. Let
	\[
	\Sigma = \left\{x \mapsto \sigma(y^T x + \theta) \colon \theta \in \r\,, y \in \r^2 \right\} \,.
	\]
	We show that $\vspan{\Sigma}$ is dense in $C^0(\bar S)$ with respect to the $C^0$ norm. Suppose by contradiction that $R = \overline{\vspan \Sigma} \subsetneq C^0(\bar S)$. Then, by the Hahn-Banach Theorem, there is $L \in (C^0(\bar S))^\ast$ such that $L \neq 0$ and $L(R) = 0$. By the Riesz Representation Theorem, there is a signed bounded regular Borel measure $\mu \neq 0$ on $\bar S$ such that
	\[
	L(h) = \int_{\bar S} h \, d\mu \quad \forall \, h \in C^0(\bar S) \,.
	\]
	Since $\sigma$ is discriminatory by Step 1, we have the desired contradiction.
	
	\textbf{Step 3}. To conclude it is sufficient to show that $\Sigma \subset \adet$. Let $\theta \in \r$ and $y \in \r^2$. We have
	\[
	\nabla_x^2 \sigma(y^T x + \theta) = \sigma''(y^T x +\theta)\begin{pmatrix}
	y_1^2 & y_1 y_2\\
	y_1 y_2 & y_2^2
	\end{pmatrix} \,.
	\]
	Thus, $\det(\nabla_x^2 \sigma(y^T x + \theta)) = 0$, concluding the proof.
\end{proof}

Lastly, we move to the proof of \Cref{lemma:affine_minimizers_zero_u}.

\begin{proof}[Proof of \Cref{lemma:affine_minimizers_zero_u}]
	The existence of $u \in W^{2, 2}(S; \r^2)$ such that $y$ is an isometric embedding is proved in \cite[Theorem 7]{FRIESECKE2006}. We are left to show that $u \in W^{2, \infty}(S; \r^2)$ and \eqref{eq:regularity_u_isometric_embedding} holds. In order to do so, we need to analyse the construction of $u$. We borrow the notation from \cite[Theorem 7]{FRIESECKE2006}. Let
	\begin{align}
		F = \sqrt{\Id - \nabla' v \otimes \nabla' v} \,,
	\end{align}
	and
	\begin{equation}
		h_F = \frac{1}{\det(F)} F^T \curl(F) \,.
	\end{equation}
	Then, $u$ is defined as $u(x') = \phi(x') - x'$, where $\phi \in W^{2, 2}(S; \r^2)$ is such that $\nabla' \phi = e^{i \theta} F$ and $\theta \in W^{1, 1}(S)$ has zero mean and satisfies $\nabla' \theta = h_F$. Here, $e^{i \theta}$ stands for the rotation matrix of angle $\theta$:
	\[
		e^{i \theta} = \begin{pmatrix}
			\cos \theta & -\sin \theta\\
			\sin \theta & \cos \theta
		\end{pmatrix} \,.
	\]
	Note that
	\begin{equation}
		\det(F) = \sqrt{\det(\Id - \nabla' v \otimes \nabla' v)} = \sqrt{1 - |\nabla' v|^2} \geq \frac{1}{2} \,.
	\end{equation}
	It is well-known that the matrix square root is differentiable and Lipschitz on the set of matrices whose determinant is positive and bounded away from $0$. Thus, $F \in W^{1, \infty}(S; \r^{2 \times 2})$ and 
	\begin{align}
		\|F\|_{\infty} & \leq C\,,\\
		\|\nabla' F\|_{\infty} & \leq C \|(\nabla')^2 v\|_{\infty}\|\nabla' v\|_{\infty} \,.
	\end{align}
	It follows that $h_F \in L^\infty(S; \r^2)$ and $\|\nabla' \theta\|_{\infty} = \|h_F\|_{\infty} \leq C\|(\nabla')^2 v\|_{\infty}\|\nabla' v\|_{\infty}$. Hence, we have
	\begin{align}
		\|(\nabla')^2 u\|_{\infty} & = \|(\nabla')^2 \phi\|_{\infty} \leq C(\|\nabla' \theta\|_{\infty}\|F\|_{\infty} + \|\nabla' F\|_{\infty}) \leq C \|(\nabla')^2 v\|_{\infty}\|\nabla' v\|_{\infty} \,,\\
		\|\nabla' u\|_{\infty} & = \|\nabla' \phi - \Id\|_{\infty} \leq C\|F - \Id\|_{\infty} + \|e^{i \theta} - \Id\|_{\infty} \\
		& \leq C\|F - \Id\|_{\infty} + \|\theta\|_{\infty} \leq C(\||\nabla' v|^2\|_{\infty} + \|\nabla' \theta\|_{\infty})\\
		& \leq C(\|(\nabla')^2 v\|_{\infty}\|\nabla' v\|_{\infty} + \|\nabla' v\|_{\infty}^2)\,,
	\end{align}
	where we have used the Poincaré--Wirtinger inequality on the term $\|\theta\|_{\infty}$ and a Taylor expansion of the matrix square root to treat the term $F - \Id$ (recall that the matrix square root has bounded derivative). Since $u$ is defined up to translation, we conclude by applying the Poincaré--Wirtinger inequality.
\end{proof}

The rest of the section is devoted to the proof of \Cref{teo:stability_false}--\ref{teo:implication_stability}.

\begin{proof}[Proof of \Cref{teo:stability_false}]
    Firstly, observe that $J_h(y_h) \leq Ch^2$. Indeed, using the test deformation \eqref{eq:test_deformation} we have
	\[
		\inf_y J_h(y) \leq J_h(\hat y_h) = h^2 \int_S f \cdot R \xprimo \, dx' \leq Ch^2 \,.
	\]
	By \Cref{teo:rigidity_estimate} there is a constant rotation $\tilde R_h \in \SO(3)$ such that
	\begin{equation}\label{eq:rigidity_deformation}
	\|\nabla_h y_h - \tilde R_h\|_{L^2}^2 \leq Ch^{-2}E_h(y_h) \,.
	\end{equation}
	We now define
	\[
		\tilde y_h = \tilde R^T_h y_h - \begin{pmatrix}
			x' \\ hx_3
		\end{pmatrix} + c_h
	\]
	where $c_h$ is chosen so that $\tilde y_h$ has zero average. By Poincaré--Wirtinger inequality, one obtains a bound from above on the elastic energy as follows
	\begin{align}
	E_h(y_h) & = J_h(y_h) + \int_\Omega f_h \cdot y_h \, dx\\
	& \leq Ch^2 + h^2 \int_\Omega f \cdot \tilde R_h \tilde y_h \, dx + h^2 \int_S f \cdot \xprimo \, dx' \\
	& \leq Ch^2 + Ch^2\|\nabla_h y_h - \tilde R_h\|_{L^2} \leq Ch^2 +Ch(E_h(y_h))^{\frac{1}{2}} \,.
	\end{align}
	Thus, by a simple application of the Young inequality, we get $E_h(y_h) \leq Ch^2$. Assume now that $R^T f \cdot e_3 \neq 0$ for some $R \in \orot$. It follows that 
	\[
		\liminf_{h \to 0} \frac{E_h(y_h)}{h^2} = e > 0 \,,
	\]
	otherwise, defining $D_h = E_h(y_h)$, by \Cref{teo:compatibility_forces} we would conclude that $R^T f \cdot e_3 = 0$ for every optimal rotation $R \in \orot$, contradicting the assumption.
    By \Cref{prop:compactness_alpha_2} there exists $\bar y \in \aiso$ such that, up to a subsequence, $y_h \to \bar y$ in $W^{1, 2}(\Omega; \r^3)$ and $\nabla_h y_h \to (\nabla' \bar y, \nu)$, where $\nu= \partial_1 \bar y \land \partial_2 \bar y$.
    By a standard $ \Gamma $-convergence argument, being $ y_{h} $ quasi-minimizing, we deduce that
    \begin{equation}
        \frac{1}{h^{2}} J_{h}(y_{h}) \to J^{\KL}(\bar y) = \inf J^{\KL} .
    \end{equation}
    In particular, since the loading term is continuous, we deduce by the $ \Gamma $-convergence of $ \frac{1}{h^{2}}E_{h} $ that
    \begin{equation}
        \frac{1}{h^{2}}E_{h}(y_{h}) \to E^{\KL}(\bar y) = e > 0.
    \end{equation}
    This implies that $\bar y \neq R \xprimo$ for every $R \in \SO(3)$, so condition \ref{item:stability_KL} is not satisfied. 
\end{proof}

We move now to the proof of \Cref{teo:stability_alternative}.

\begin{proof}[Proof of \Cref{teo:stability_alternative}]
	The proof follows the steps of \cite[Theorem 4]{LECUMBERRY2009}.
    Arguing as in the proof of \Cref{teo:stability_false} we get $E_h(y_h) \leq Ch^2$.
	
	\textbf{Step 1}. Firstly, suppose by contradiction that $h^{-2} E_h(y_h) \to e > 0$. In this case we can argue as in the proof of \Cref{teo:stability_false} to deduce that $y_h \to \bar y$ in $W^{1,2}(S; \r^3)$, $E^{\KL}(\bar y) = e > 0$ and $J^{\KL}(\bar y) = 0$ contradicting the stability condition \ref{item:stability_KL}. 
	
	\textbf{Step 2}. Suppose now that $h^{-2} E_h(y_h) \to 0$ and $h^{-4} E_h(y_h) \to +\infty$. We will show that this gives a contradiction. Set $D_h = E_h(y_h)$ and apply \Cref{prop:compactness_alpha_greater_2} to construct the sequences $R_h, u_h$ and $v_h$ and their corresponding limits. By \Cref{prop:limit_optimal_rotation} $R_h \to \bar R \in \orot$ thus, at least for $h \ll 1$, the projection $P(R_h)$ of $R_h$ onto $\orot$ is well-defined. Define $d_h = \dist_{\SO(3)}(R_h, \orot)$ (see \eqref{eq:definition_intrinsic_distance} for the definition of $\dist_{\SO(3)}$). Let $W_h \in \normal_{P(R_h)}$ be such that $|W_h| = 1$ and $R_h = P(R_h)e^{d_h W_h}$. Recall that $\normal_{P(R_h)}$ is the normal space to $\orot$ at the point $P(R_h)$ (see \eqref{eq:definition_normal} for the definition). Clearly, up to a subsequence, $W_h \to \bar W_1$ with $|\bar W_1| = 1$. Moreover, by the definition of $\normal$ it is easy to prove that $\bar W_1 \in \normal_{\bar R}$.
	
	We now show that $d_h = O(h^{-1}\sqrt{D_h})$. Let $v \in C^\infty(\bar S)$ with $\det((\nabla')^2v) = 0$ in $S$ and $\tilde u_h \in W^{2, \infty}(S)$ given by \Cref{lemma:contruction_isometries} so that the map
	\[
		\tilde y_h(x') = \xprimo + \begin{pmatrix}
		h^{-2}D_h \tilde u_h \\
		h^{-1}\sqrt{D_h}v
		\end{pmatrix} \,
	\]
	is an isometric embedding.
    Note that, since $ h^{-1} \sqrt{D_{h}} \to 0 $, we have the uniform bound $ \|\tilde u_{h}\|_{W^{2, \infty}} \leq C $.
    Consider the test deformation 
    \begin{equation}
		\hat y_h = P(R_{h})\tilde y_h + h x_3 P(R_{h})\nu_h \,,
	\end{equation}
    By \Cref{prop:fine_test_deformation_det} we have $ E_{h}(\hat y_{h}) = O(D_{h}) $.
    Thus, 
	\begin{equation} \label{eq:bound_dh_det_1}
		\begin{aligned}
		& J_h(y_h) - J_h(\hat y_h) \geq \int_\Omega f_h \cdot \hat y_h \, dx - \int_\Omega f_h \cdot y_h \, dx + O(D_h) \\
		& \quad = - \int_S f \cdot R_h \begin{pmatrix}
		D_h u_h \\ h \sqrt{D_h} v_h
		\end{pmatrix} \, dx' + h^2 \int_S f \cdot (P(R_h) - R_h) \xprimo \, dx' \\
		& \quad \hphantom{\geq} \, \, +\int_S f \cdot P(R_h) \begin{pmatrix}
		D_h \tilde u_h \\ h \sqrt{D_h} v
		\end{pmatrix} \, dx' + O(D_h) \,.
		\end{aligned}
	\end{equation}
	As showed in \eqref{eq:condition_maximality}, we have that
	\begin{equation}
	\int_S f \cdot P(R_h) W \xprimo \, dx' = 0 \quad \forall \, W \in \rskew \,. \label{eq:zero_skew}
	\end{equation}
	Expanding the exponential $e^{d_h W_h}$, recalling that by \Cref{teo:compatibility_forces} we have $P(R_h)^T f \cdot e_3 = 0$, and using both \eqref{eq:regularity_u_isometric_embedding} and  \eqref{eq:zero_skew}, we get from \eqref{eq:bound_dh_det_1}
	\begin{equation} \label{eq:bound_dh_det}
	\begin{aligned}
	& J_h(y_h) - J_h(\hat y_h) \geq - h^2 d_h^2 \int_S f \cdot P(R_h)W_h^2 \xprimo \, dx' \\
	& \quad \hphantom{=} \, \, + h\sqrt{D_h} \int_S f \cdot (P(R_h) - R_h) \lmatrix{v_h} \, dx' + O(D_h, h^2 d_h^3) \\
	& \quad \geq - h^2 d_h^2 \int_S f \cdot P(R_h)W_h^2 \xprimo \, dx' \\
	& \quad \hphantom{=} \, \, - hd_h\sqrt{D_h} \int_S f \cdot P(R_h)W_h \lmatrix{v_h} \, dx' + O(D_h, h^2 d_h^3, h \sqrt{D_h}d_h^2)\,.
	\end{aligned}
	\end{equation}
	Suppose by contradiction that $\frac{h d_h}{\sqrt{D_h}} \to +\infty$. Then dividing \eqref{eq:bound_dh_det} by $h^2 d_h^2$ we get
	\begin{align}
		& \frac{1}{h^2d_h^2}(J_h(y_h) - J_h(\hat y_h)) \geq - \int_S f \cdot P(R_h) W_h^2 \xprimo \, dx' \\
		& \quad  - \frac{\sqrt{D_h}}{hd_h} \int_S f \cdot P(R_h) W_h \lmatrix{v_h} \, dx' + O\left(d_h, \frac{D_h}{h^2d_h^2}, \frac{\sqrt{D_h}}{h}\right)\,. \label{eq:almost_contradiction_dh_det}
	\end{align}
	Note that, by \eqref{eq:small_o_minimizing} we have
	\[
		\limsup_{h \to 0} \frac{1}{h^2 d_h^2}(J_h(y_h) - J_h(\hat y_h))= \limsup_{h \to 0} \frac{D_h}{h^2 d_h^2} \dfrac{h^4}{D_h} \frac{1}{h^4} (J_h(y_h) - J_h(\hat y_h)) \leq 0 \,.
	\]
	Passing to the limit in \eqref{eq:almost_contradiction_dh_det} we deduce that
	\[
		0 \geq - \int_S f \cdot \bar R \bar W_1^2 \xprimo \, dx' > 0 \,,
	\]
	where the last inequality follows from the fact that $\bar W_1 \in \normal_{\bar R}$ and $\bar W_1 \neq 0$ (see \eqref{eq:definition_normal}). This gives a contradiction and proves that $d_h = O(h^{-1} \sqrt{D_h})$. 

	To conclude the proof of Step 2 we show now that condition \ref{item:stability_VK} is violated, getting a contradiction. Set
	\[
		\bar W = \lim_{h \to 0} \frac{h}{\sqrt{D_h}} d_h W_h \,.
	\]
	Since $\bar W_1 \in \normal_{\bar R}$ we have $\bar W \in \normal_{\bar R}$. We have that
	\begin{equation}
		\begin{aligned}
			\dfrac{1}{D_h}(J_h(y_h) + h^2F(P(R_h))) & = \dfrac{1}{D_h}E_h(y_h) - \frac{h^2}{D_h} \int_\Omega f \cdot y_h \, dx' \\
			& \hphantom{=} \, \,  + \frac{h^2}{D_h}\int_S f \cdot P(R_h) \xprimo \, dx'\\
			& = \dfrac{1}{D_h}E_h(y_h) + \frac{h^2}{D_h}\int_S f \cdot (P(R_h) - R_h) \xprimo \, dx'\\
			& \hphantom{=} \, \, - \frac{h^2}{D_h}\int_S f \cdot R_h \begin{pmatrix}
				h^{-2}D_h u_h \\ h^{-1}\sqrt{D_h} v_h
			\end{pmatrix} \, dx' \,.
		\end{aligned}
	\end{equation}
	Expanding twice the exponential map, recalling that $P(R_h)^T f \cdot e_3 = 0$, and by \eqref{eq:zero_skew} we get 
	\begin{equation}
		\begin{aligned}
			& \dfrac{1}{D_h}(J_h(y_h) + h^2F(P(R_h))) =   \dfrac{1}{D_h}E_h(y_h) \\
			& \quad \hphantom{=} \, \, - \frac{h^2d_h^2}{D_h} \int_S f \cdot P(R_h)W_h^2 \xprimo \, dx' - \int_S f \cdot R_h \umatrix{u_h} \, dx' \\
			& \quad \hphantom{=} \, \,  + \frac{h}{\sqrt{D_h}} \int_S f \cdot (P(R_h) - R_h) \lmatrix{v_h} \, dx' + O\left(\dfrac{h^2d_h^3}{D_h}\right)\\
			& \quad = \dfrac{1}{D_h}E_h(y_h) - \frac{h^2d_h^2}{D_h} \int_S f \cdot P(R_h)W_h^2 \xprimo \, dx' - \int_S f \cdot R_h \umatrix{u_h} \, dx'\\
			& \quad \hphantom{=} \, \, - \frac{hd_h}{\sqrt{D_h}} \int_S f \cdot P(R_h)W_h \lmatrix{v_h} \, dx' + O\left(\dfrac{h^2d_h^3}{D_h}, \dfrac{hd_h^2}{\sqrt{D_h}}\right) \,.
		\end{aligned}
	\end{equation}
	We denote by $\bar v$ and $\bar u$ the limits of $v_h$ and $u_h$, respectively. Note that by \Cref{prop:compactness_alpha_greater_2}, $(\bar u, \bar v) \in \aisolin$. Since by definition $\frac{1}{D_h}E_h(y_h) \to 1$, passing to the limit we get by \Cref{teo:gamma_convergence_alpha_greater_2}--\ref{item:liminf_inequality}
	\begin{equation}\label{eq:contradiction_det}
		\begin{aligned}
			& \liminf_{h \to 0} \dfrac{1}{D_h}(J_h(y_h) + h^2 F(P(R_h)))  \\
			& \quad = 1 - \int_S f \cdot \bar R \umatrix{\bar u} \, dx' - \int_S f \cdot \bar R \bar W \lmatrix{\bar v} \, dx' \\
			& \quad \hphantom{=} \, \, - \int_S f \cdot \bar R \bar W^2 \xprimo \, dx' \geq J^{\VK}(\bar u, \bar v, \bar R, \bar W) \geq 0 \,,
		\end{aligned}
	\end{equation}
	where the last inequality follows from \ref{item:stability_VK}. 
	
	Let $\hat y_h$ be the test deformation in \eqref{eq:fine_test_deformation} with $v \in C^\infty(\bar S)$ and $R \in \orot$.
    By \Cref{prop:fine_test_deformation}, we have that $ E_{h}(\hat y_{h}) = O(h^{4}) $, hence 
	\begin{align}
		\dfrac{1}{D_h} (J_h(\hat y_h) + h^2F(P(R_h))) & \leq - \frac{h^3}{D_h} \int_S f \cdot R \lmatrix{v} \, dx' + O\left(\frac{h^4}{D_h}\right) \\
		& = O\left(\frac{h^4}{D_h}\right) \to 0 \,,
	\end{align}
	where we used that $F(P(R_h)) = F(R)$ for every $R \in \orot$ and that $R^T f \cdot e_3 = 0$. In particular, by the quasi-minimizing property of $y_h$
	\[
		\limsup_{h \to 0} \dfrac{1}{D_h}(J_h(y_h) + h^2 F(P(R_h))) \leq \limsup_{h \to 0} \dfrac{1}{D_h}(J_h(y_h) - J_h(\hat y_h)) = 0 \,.
	\]
	Hence, all the inequalities in \eqref{eq:contradiction_det} are in fact equalities, and we have both $J^{\VK}(\bar u, \bar v, \bar R, \bar W) = 0$ and $E^{\VK}(\bar u, \bar v) = 1$. Since $(\bar u, \bar v) \in \aisolin$, this contradicts \ref{item:stability_VK}.
	 
	\textbf{Step 3}. By the previous steps, we obtain that $E_h(y_h) \leq Ch^{4}$.
    We now apply \Cref{prop:compactness_alpha_greater_2}  with $D_h = h^4$ to construct the sequences $R_h, u_h, v_h$.
    Define $d_h$ and $W_h$ as in Step 2.
    We prove now that $d_h = O(h)$.
    The argument is similar to the one already seen.
    Consider the test deformation \eqref{eq:fine_test_deformation} with $v \in C^\infty(\bar S)$ and $R = P(R_h)$.
    By \Cref{prop:fine_test_deformation} we have $E_h(\hat y_h) = O(h^4)$, thus, expanding the exponential and recalling that $F(RW) = 0$ for every $R \in \orot$ and $W \in \rskew$
	\begin{equation} \label{eq:bound_dh}
		\begin{aligned}
		& J_h(y_h) - J_h(\hat y_h) \geq \int_\Omega f_h \cdot \hat y_h \, dx - \int_\Omega f_h \cdot y_h \, dx + O(h^4)\\
		& \quad = - h^2 \int_S f \cdot R_h \begin{pmatrix}
		h^2 u_h \\ hv_h
		\end{pmatrix} \, dx' + h^2 \int_S f \cdot (P(R_h) - R_h) \xprimo \, dx'\\
		& \quad \hphantom{\geq} \, \, + h^2 \int_S f \cdot P(R_h) \lmatrix{hv} \, dx' + O(h^4) \\
		& \quad \geq - h^2 d_h^2 \int_S f \cdot P(R_h)W_h^2 \xprimo \, dx' \\
		& \quad \hphantom{=} \, \, + h^3 \int_S f \cdot (P(R_h) - R_h) \lmatrix{v_h} \, dx' + O(h^4, h^2 d_h^3) \\
		& \quad \geq - h^2 d_h^2 \int_S f \cdot P(R_h)W_h^2 \xprimo \, dx' \\
		& \quad \hphantom{=} \, \, - h^3d_h \int_S f \cdot P(R_h)W_h \lmatrix{v_h} \, dx' + O(h^4, h^2 d_h^3, h^3d_h^2)\,.
		\end{aligned}
	\end{equation}
	Suppose by contradiction that $\frac{d_h}{h} \to +\infty$. Then, dividing \eqref{eq:bound_dh} by $h^2d_h^2$ and passing to the limit we deduce that
	\[
		0 \geq \int_Sf \cdot \bar R \bar W_1^2 \xprimo \, dx' > 0 \,,
	\]
	where the last inequality follows from the fact that $0 \neq \bar W_1 \in \normal_{\bar R}$. This provides the desired contradiction. 
	
	Define as before
	\[
		\bar W = \lim_{h \to 0} \frac{d_h}{h} W_h \,.
	\]
	Finally, expanding the exponential map
	\begin{equation}
		\begin{aligned}
			&\dfrac{1}{h^4}(J_h(y_h) + h^2 F(P(R_h))) =  \\
			& \quad = \dfrac{1}{h^4}E_h(y_h) - \frac{1}{h^2} \int_\Omega f \cdot y_h \, dx' + \frac{1}{h^2}\int_S f \cdot P(R_h) \xprimo \, dx'\\
			& \quad = \dfrac{1}{h^4}E_h(y_h) + \frac{1}{h^2}\int_S f \cdot (P(R_h) - R_h) \xprimo \, dx'  - \frac{1}{h^2}\int_S f \cdot R_h \begin{pmatrix}
			h^2 u_h \\ h v_h
			\end{pmatrix} \, dx'\\
			& \quad =  \dfrac{1}{h^4}E_h(y_h) - \frac{d_h^2}{h^2} \int_S f \cdot P(R_h)W_h^2 \xprimo \, dx' - \int_S f \cdot R_h \umatrix{u_h} \, dx'\\
			& \quad \hphantom{=} \, \,  + \frac{1}{h} \int_S f \cdot (P(R_h) - R_h) \lmatrix{v_h} \, dx' + O\left(\frac{d_h^3}{h^2}\right) \\
			& \quad = \dfrac{1}{h^4}E_h(y_h) - \frac{d_h^2}{h^2} \int_S f \cdot P(R_h)W_h^2 \xprimo \, dx'\\
			& \quad \hphantom{=} \, \, - \int_S f \cdot R_h \umatrix{u_h} \, dx' - \frac{d_h}{h} \int_S f \cdot P(R_h)W_h \lmatrix{v_h} \, dx' + O\left(\frac{d_h^3}{h^2}, \frac{d_h^2}{h}\right) \,,
		\end{aligned}
	\end{equation}
	so that
	\begin{equation}
		\liminf_{h \to 0} \dfrac{1}{h^4}(J_h(y_h) + h^2F(P(R_h))) \geq J^{\VK}(\bar u, \bar v, \bar R, \bar W) \,.
	\end{equation}
	
	Let $(u, v, R, W)$ be an admissible quadruplet. Construct a recovery sequence $(\tilde y_h)$ for $u$ and $v$ as in \Cref{teo:gamma_convergence_alpha_greater_2}--\ref{item:recovery_sequence_alpha_4}. The sequences of rescaled displacements for the recovery sequence, defined as in \eqref{eq:definition_u}--\eqref{eq:definition_v}, will be denoted by $\tilde u_h$ and $\tilde v_h$. We have
	\begin{equation}
		\begin{aligned}
			J^{\VK}(\bar u, \bar v, \bar R, \bar W) & \leq \liminf_{h \to 0} \dfrac{1}{h^4}(J_h(y_h) + h^2 F(P(R_h))) \\
			& \leq \limsup_{h \to 0} \dfrac{1}{h^4}(\inf_y J_h(y) + h^2 F(P(R_h))) \\
			& \leq \limsup_{h \to 0} \dfrac{1}{h^4}(J_h(Re^{hW}\tilde y_h) + h^2 F(R)) \,.
		\end{aligned}
	\end{equation}
	To conclude it is sufficient to prove that 
	\[
		\limsup_{h \to 0} \dfrac{1}{h^4}(J_h(Re^{hW}\tilde y_h) + F(R)) = J^{\VK}(u, v, R, W) \,.
	\]
	Expanding the expression of $J_h$ we have
	\begin{equation}
		\begin{aligned}
			& \dfrac{1}{h^4}(J_h(Re^{hW}\tilde y_h) + h^2 F(R)) = \dfrac{1}{h^4}E_h(Re^{hW}\tilde y_h) - \dfrac{1}{h^2} \int_\Omega f \cdot Re^{hW} \tilde y_h \, dx \\
			& \, \, \, \hphantom{=} \, \, + \dfrac{1}{h^2} \int_S f \cdot R \xprimo \, dx' \\
			& \, \, \, = \dfrac{1}{h^4}E_h(\tilde y_h) - \dfrac{1}{h^2} \int_S f \cdot Re^{hW} \begin{pmatrix}
				h^2 \tilde u_h \\ h\tilde v_h
			\end{pmatrix} \, dx' + \dfrac{1}{h^2} \int_S f \cdot (R - Re^{hW}) \xprimo \, dx'\\
			& \, \, \, = \dfrac{1}{h^4}E_h(\tilde y_h) - \int_S f \cdot Re^{hW} \umatrix{\tilde u_h} \, dx' - \int_S f \cdot RW \lmatrix{\tilde v_h} \, dx' \\
			& \, \, \, \hphantom{=} \, \, - \int_S f \cdot RW^2 \xprimo \, dx' +O(h) \to J^{\VK}(u, v, R, W) \,,
		\end{aligned}
	\end{equation}
	concluding the proof of the minimality.
\end{proof}

We conclude the section by proving \Cref{teo:implication_stability}.

\begin{proof}[Proof of \Cref{teo:implication_stability}]
	Suppose by contradiction that there exists an admissible quadruplet $(\bar u,\bar v,\bar R,\bar W)$ such that $(\bar u,\bar v) \in \aisolin$ and $J^{\VK}(\bar u,\bar v,\bar R,\bar W) < 0$. Let $\delta > 0$ and $\tilde v \in C^\infty(\bar S)$ such that $\|\bar v- \tilde v\|_{W^{2, 2}} \leq \delta$. Let $1 \gg \varepsilon > 0$. By \cite[Theorem 7]{FRIESECKE2006}, there is $u_\varepsilon \in W^{2, 2}(S; \r^2)$ such that
	\[
		y_\varepsilon(x') = \begin{pmatrix}
			x' + \varepsilon^2 u_\varepsilon \\ \varepsilon \tilde v
		\end{pmatrix}
	\]
	is an isometric embedding and
	\begin{equation}\label{eq:uepsilon_estimate}
		\|u_\varepsilon\|_{W^{2
		, 2}} \leq C\left(\|\nabla' \tilde v\|_{L^\infty}\|(\nabla')^2 \tilde v\|_{L^2} + \|\nabla' \tilde v\|_{L^2}^2\right) \,.
	\end{equation}
	It follows that along a non-relabelled subsequence $u_\varepsilon \todeb u$ in $W^{2, 2}(S; \r^2)$ for some $u \in W^{2, 2}(S; \r^2)$. Moreover, since $\nabla' y_\varepsilon^T \nabla' y_\varepsilon = \Id$, we have
	\[
		0 = \varepsilon^2 \left(\nabla' u_\varepsilon^T + \nabla' u_\varepsilon + \nabla' \tilde v \otimes \nabla' \tilde v\right) + o(\varepsilon^2) \,,
	\]
	where $o(\varepsilon^2)$ has to be intended in the $L^2$ sense. Dividing by $\varepsilon^2$ and passing to the limit we deduce that $(u, \tilde v) \in \aisolin$. Moreover,
	\begin{equation}
		\begin{aligned}
			 \sym(\nabla' u - \nabla' \bar u) & = 2 (\nabla' \tilde v \otimes \nabla' \tilde v - \nabla' \bar v \otimes \nabla' \bar v) \\
			 & = 2 (\nabla' (\tilde v - \bar v) \otimes \nabla' \tilde v - \nabla' \bar v \otimes \nabla' (\bar v - \tilde v)) \,.
		\end{aligned}
	\end{equation}
	Hence, by Korn's inequality, there exists $A \in \r^{2 \times 2}_{\skw}$ and $\eta \in \r^{2}$ such that
	\begin{equation}\label{eq:estimate_u_korn}
		\|u - \bar u - Ax' - \eta\|_{L^2} \leq C\delta \,.
	\end{equation}
	Consider the deformation
	\[
		\bar y_\varepsilon(x') = \bar R e^{\varepsilon \bar W} y_\varepsilon \in \aiso \,.
	\]
	We have
	\[
		\nabla' \bar y_\varepsilon = \bar R e^{\varepsilon \bar W}\left( \begin{pmatrix}
			e_1 & e_2
		\end{pmatrix} + \begin{pmatrix}
			\varepsilon^2 \nabla' u_\varepsilon \\
			\varepsilon \nabla' \tilde v
		\end{pmatrix} \right)
	\]
	and
	\[
		\nu_\varepsilon = \partial_1 \bar y_\varepsilon \land \partial_2 \bar y_\varepsilon = \bar R e^{\varepsilon \bar W} \left(e_3 - \varepsilon \begin{pmatrix}
		\nabla' \tilde v^T \\ 0
		\end{pmatrix} \right) +O(\varepsilon^2) \,.
	\]
	It follows that
	\[
		\nabla' \nu_\varepsilon = - \varepsilon\bar R e^{\varepsilon\bar W} \begin{pmatrix}
            (\nabla')^2 \tilde v \\
			0 
		\end{pmatrix} + O(\varepsilon^2)
	\]
	and
	\[
		(\nabla' \bar y_\varepsilon)^T \nabla' \nu_\varepsilon = - \varepsilon (\nabla')^{2} \tilde v + O(\varepsilon^2) \,.
	\]
	Thus, by condition \ref{item:stability_KL},
	\begin{equation}
		\begin{aligned}
			0 & \leq J^{\KL}(\bar y_\varepsilon) = \int_S \bar Q((\nabla' \bar y_\varepsilon)^T \nabla' \nu_\varepsilon) \, dx' - \int_S f \cdot \bar y_\varepsilon \, dx' \\
			& = \varepsilon^2 \int_S \bar Q((\nabla')^2 \tilde v) \, dx' - \int_S f \cdot \bar R e^{\varepsilon \bar W} \xprimo \, dx' - \varepsilon \int_S f \cdot \bar R e^{\varepsilon\bar W} \lmatrix{\tilde v} \, dx' \\
			& \hphantom{=} \, \, - \varepsilon^2 \int_S f \cdot \bar R e^{\varepsilon \bar W} \umatrix{u_\varepsilon} \, dx' +o(\varepsilon^2) \,.
		\end{aligned}
	\end{equation}
	By \Cref{teo:stability_false} we have $(\bar R)^T f \cdot e_3 = 0$. Expanding the exponential around the identity and recalling that $F(\bar RW) = 0$ for every $W \in \rskew$, we get
	\begin{equation}
		\begin{aligned}
			0 & \leq J^{\KL}(\bar y_\varepsilon) \leq \varepsilon^2 \int_S \bar Q((\nabla')^2 \tilde v) \, dx'  - F(\bar R) - \varepsilon^2 \int_S f \cdot \bar R (\bar W)^2 \xprimo \, dx' \\
			& \hphantom{=} \, \, - \varepsilon^2 \int_S f \cdot \bar R \bar W \lmatrix{\tilde v} \, dx' - \varepsilon^2\int_S f \cdot \bar R\umatrix{u_\varepsilon} \, dx'+ o(\varepsilon^2) \,.
		\end{aligned}
	\end{equation}
	Dividing by $\varepsilon^2$ and using the fact that $F(\bar R) \geq 0$ by \Cref{lemma:sign_F}, passing to the limit we deduce that $0 \leq J^{\VK}(u, \tilde v, \bar R, \bar W)$. Hence, by definition of $\tilde v$ and \eqref{eq:estimate_u_korn} we get
	\begin{equation}
		\begin{aligned}
			0 & \leq J^{\VK}(\bar u,\bar v,\bar R,\bar W) + \int_S f \cdot \bar R \umatrix{Ax' + \eta} \, dx' + C\delta \\
			& = J^{\VK}(\bar u,\bar v,\bar R,\bar W) + C\delta \,,
		\end{aligned}
	\end{equation}
	where in the last equality we have used \eqref{eq:mean_condition_forces} and the fact that $F(\bar RW) = 0$ for every $W \in \rskew$. Since $\delta$ is arbitrary we reach a contradiction. 
	
	We now prove that \ref{item:stability_VK} holds for $J^{\VK}_\varepsilon$. Suppose that there is an admissible quadruplet $(\bar u,\bar v,\bar R,\bar W)$ such that $(\bar u,\bar v) \in \aisolin$ and $J^{\VK}_\varepsilon(\bar u,\bar v,\bar R,\bar W) \leq 0$ for some $\varepsilon > 0$. We will show that $\bar v$ is affine. Let
	\[
		K = \int_S f \cdot\bar R \umatrix{\bar u} \, dx' + \int_S f \cdot\bar R \bar W \lmatrix{\bar v} \, dx' + \int_S\bar R(\bar W)^2 \xprimo \, dx' \,.
	\]
	If $K \leq 0$, since
	\[
		0 \geq J^{\VK}_\varepsilon(\bar u,\bar v,\bar R,\bar W) = E^{\VK}(\bar u,\bar v) - (1-\varepsilon)K \geq E^{\VK}(\bar u,\bar v) \,,
	\]
	we get that $E^{\VK}(\bar u,\bar v) = 0$, thus, $\bar v$ is affine. Conversely, if $K > 0$ we deduce that
	\[
		J^{\VK}(\bar u,\bar v,\bar R,\bar W) = J^{\VK}_\varepsilon(\bar u,\bar v,\bar R,\bar W) - \varepsilon K < 0 \,,
	\]
	which gives a contradiction.
\end{proof}

%% file: minima.tex
\section{Attainment of the infimum of \texorpdfstring{$\bm{J^{\VK}}$}{JVK}}\label{section:minimum}

In this last section, we will prove \Cref{teo:JVK_has_minimum}. The stability condition \ref{item:stability_VK} assures that all configurations in $\aisolin$ with zero total energy have zero Von K\'arm\'an elastic energy, i.e., $v$ is affine. However, we do not expect that all affine functions have zero total energy, unless $f = 0$. In the following series of results, we study the specific structure of such affine minimizers. We recall that we assume $f$ not to be identically zero.
Given an optimal rotation $ R \in \orot $, in the following results we will often use the coefficients $ a(R) $, $ b(R) $, and $ c(R) $ as defined in \eqref{eq:definition_a}--\eqref{eq:definition_c}.

\begin{prop}\label{prop:structure_affine_minimizers}
	Suppose that \ref{item:stability_VK} holds, $R^T f \cdot e_3 = 0$ for every $R \in \orot$, and $\dim \orot = 1$. Let $(u, v, R, W)$ be an admissible quadruplet such that $(u, v) \in \aisolin$ and $J^{\VK}(u, v, R, W) = 0$. Then $W = 0$ and there are $\lambda, \delta \in \r$, $\eta \in \r^2$, and $A \in \r^{2 \times 2}_{\skw}$ such that, if $a(R) > 0$, then
	\begin{align}
		v(x') & = - \lambda \dfrac{c(R)}{a(R)} x_1 + \lambda x_2 + \delta \,, \\
		u(x') & = -\frac{\lambda^2}{2}\begin{pmatrix}
			\dfrac{b(R)}{a(R)} x_1 - \dfrac{c(R)}{a(R)} x_2\\
			- \dfrac{c(R)}{a(R)} x_1 + x_2
		\end{pmatrix} + A x' + \eta \,,
	\end{align}
	whereas, if $b(R) > 0$, then
	\begin{align}
		v(x') & = \lambda x_1 - \lambda \dfrac{c(R)}{b(R)} x_2 + \delta \,, \\
		u(x') & = -\frac{\lambda^2}{2}\begin{pmatrix}
		x_1 - \dfrac{c(R)}{b(R)} x_2\\
		- \dfrac{c(R)}{b(R)} x_1 + \dfrac{a(R)}{b(R)} x_2
		\end{pmatrix} + A x' + \eta \,.
	\end{align}
\end{prop}

\begin{proof}
	The stability condition \ref{item:stability_VK} implies that $v = \lambda_1 x_1 + \lambda_2 x_2 + \delta$. Since $(u, v) \in \aisolin$, we deduce that
	\[
		u(x') = -\frac{1}{2} \begin{pmatrix}
			\lambda_1^2 & \lambda_1 \lambda_2\\
			\lambda_1 \lambda_2 & \lambda_2^2
		\end{pmatrix}\begin{pmatrix}
			x_1 \\ x_2
		\end{pmatrix} + Ax' + \eta \,,
	\] 
	for some $\eta \in \r^2$ and $A \in \r^{2 \times 2}_{\skw}$. Now for any $\bar A \in \r^{2 \times 2}_{\skw}$, $\bar \eta \in \r^2$, and $\bar \delta \in \r$ we have 
	\[
		J^{\VK}(u + \bar Ax' + \bar \eta, v + \bar \delta, R, W) = J^{\VK}(u, v, R, W) \,.
	\]
	This follows from assumption \eqref{eq:mean_condition_forces} and the fact that $F(RW) = 0$ for any $W \in \rskew$. In particular, we can suppose $A, \delta$, and $\eta$ to be zero.

	Suppose $a(R) \neq 0$ (the proof for the case $b(R) \neq 0$ is analogous). We will write $a, b, c$ in place of $a(R), b(R), c(R)$ in order to streamline the exposition. By \Cref{cor:structure_normal} in this case $W$ is of the form
	\[
		W = \begin{pmatrix}
			0 & W_{12} & W_{13}\\
			-W_{12} & 0 & \dfrac{c}{a}W_{13}\\
			-W_{13} & -\dfrac{c}{a}W_{13} & 0
		\end{pmatrix} \,.
	\]
	Let us define $K(W) = F(RW^2)$ and $J_{\min} = J^{\VK}(u, v, R, W)$. With some simple expansion (recall that $ab - c^2 = 0$ by \Cref{prop:condition_coefficients_abc}, since $f \neq 0$) we have
	\begin{equation}
		\begin{aligned}
			J_{\min} & = \dfrac{1}{2} \int_S f \cdot R \begin{pmatrix}
				\lambda_1^2 x_1 + \lambda_1 \lambda_2 x_2\\
				\lambda_1 \lambda_2 x_1 + \lambda_2^2 x_2\\
				0
			\end{pmatrix} \, dx'\\
			& \hphantom{=} \, \, - \int_S f \cdot R \begin{pmatrix}
			0 & W_{12} & W_{13}\\
			-W_{12} & 0 & \frac{c}{a}W_{13}\\
			-W_{13} & -\frac{c}{a}W_{13} & 0
			\end{pmatrix} \lmatrix{\lambda_1 x_1 + \lambda_2 x_2} \, dx' - K(W) \\
			& = \dfrac{1}{2}(\lambda_1^2a + 2\lambda_1 \lambda_2 c + \lambda_2^2 b) - \lambda_1W_{13}(a + b) - \lambda_2W_{13}c\left(1 + \frac{b}{a}\right) - K(W) \,.
		\end{aligned}
	\end{equation}
	If we define
	\[
		M = \begin{pmatrix}
			a & c \\
			c & b
		\end{pmatrix}\,, \quad B = W_{13} \begin{pmatrix}
			a + b\\
			c\left(1 + \frac{b}{a}\right)
		\end{pmatrix}\,, \quad \Lambda= \begin{pmatrix}
			\lambda_1 \\ \lambda_2
		\end{pmatrix} \,,
	\]
	then we have
	\[
		J_{\min} = \dfrac{1}{2} \Lambda^T M \Lambda - B  \Lambda - K(W) \,.
	\]
	By \Cref{lemma:a_b_positive} and \Cref{prop:condition_coefficients_abc} $M$ is positive semidefinite and by \ref{item:stability_VK} $ \Lambda $ is a minimizer of the map  
    \begin{equation}
        v \mapsto \frac{1}{2} v^{T} M v - Bv - K(W).
    \end{equation}
    Thus, $M\Lambda = B$. Solving this system one easily gets that
	\[
		\lambda_1= -\frac{c}{a} \lambda_2 + W_{13}\left(1+ \frac{b}{a}\right) \,.
	\]
	To conclude we just need to prove that $W = 0$. Observe that 
	\[
		(W^2)' = - \begin{pmatrix}
			W_{12}^2 + W_{13}^2 & \dfrac{c}{a}W_{13}^2\\
			\dfrac{c}{a}W_{13}^2 & W_{12}^2 + \dfrac{b}{a}W_{13}^2
		\end{pmatrix}\,.
	\]
	Thus, by definition of $K(W)$,
	\[
		K(W) = -(W_{12}^2 + W_{13}^2)a -2bW_{13}^2 -bW_{12}^2 - \dfrac{b^2}{a}W_{13}^2 \,.
	\]
	Substituting the expression of $\lambda_1$ and $K(W)$ in $J_{\min}$ we get
	\[
		J_{\min} = W_{12}^2(a+b) + W_{13}^2\dfrac{1}{2a}(a+b)^2 \,,
	\]
	so that, since $ a + b > 0 $ and $J_{\min} = 0$, we deduce $W = 0$.
\end{proof}

To simplify the exposition, given $f$ such that $R^T f \cdot e_3 = 0$ for every $R$ optimal rotation $R \in \orot$ let us define
\begin{align}
	& \mathcal{V}_R = \begin{cases}
		\left\{v \in W^{2,2}(S) \colon v(x') = -\lambda \dfrac{c(R)}{a(R)} x_1 + \lambda x_2\,, \, \,  \lambda \in \r \right\} & \text{if} \, \, a(R) \neq 0 \,,\\\\
		\left\{v \in W^{2, 2}(S) \colon v(x') = \lambda x_1 - \lambda \dfrac{c(R)}{b(R)} x_2\,, \, \, \lambda \in \r \right\} & \text{if} \, \, b(R) \neq 0 \,,
	\end{cases}\\\\
	& \mathcal{U}_R = \left\{u \in W^{1, 2}(S; \r^2) \colon u(x') = - \frac{1}{2} (\nabla' v \otimes \nabla' v) x'\,, v \in \mathcal{V}_R \right\} \,.
\end{align}

\begin{lemma}\label{lemma:affine_minimizers_zero_u}
	Suppose $R^T f \cdot e_3 = 0$ for every $R \in \orot$ and $\dim \orot = 1$. Let $R \in \orot$. Then
	\begin{equation}
		\int_S f \cdot R \umatrix{u} \, dx' = 0\,,
	\end{equation}
	for every $u \in \mathcal{U}_R$.
\end{lemma}

\begin{proof}
	Let $u \in \mathcal{U}_R$ and $v \in \mathcal{V}_R$ be such that $u(x') = - \frac{1}{2} (\nabla' v \otimes \nabla' v) x'$. By \eqref{eq:definition_tangent} it is sufficient to prove that $\nabla' v \otimes \nabla' v = - (W^2)'$ for some $W \in \tang_R$. 
	
	Suppose $a(R) \neq 0$. Then $v(x') = -\lambda \dfrac{c(R)}{a(R)} x_1 + \lambda x_2$ for some $\lambda \in \r$, so
	\[
		\nabla' v \otimes \nabla' v = \lambda^2\begin{pmatrix}
			\frac{b(R)}{a(R)} & -\frac{c(R)}{a(R)}\\
			-\frac{c(R)}{a(R)} & 1
		\end{pmatrix} \,,
	\]
	where we used \Cref{prop:condition_coefficients_abc}. Then, defining
	\[
		W = \lambda \begin{pmatrix}
			0 & 0 & \frac{c(R)}{a(R)}\\
			0 & 0 & -1\\
			-\frac{c(R)}{a(R)} & 1 & 0
		\end{pmatrix} \,,
	\]
	we easily get $\nabla' v \otimes \nabla' v = - (W^2)'$ and $W \in \tang_R$ by \Cref{prop:condition_coefficients_abc}. The case $b(R) \neq 0$ can be treated similarly.
\end{proof}

\begin{lemma}\label{lemma:definition_xi}
	Suppose that $R^T f \cdot e_3 = 0$ for every $R \in \orot$ and $\dim \orot = 1$. Let $R \in \orot$ and $v \in \mathcal{V}_R$. Then for any $\bar v \in W^{2, 2}(S)$ there is $\xi \in W^{1, 2}(S; \r^2)$ such that
	\begin{equation}
		\nabla' \xi^T + \nabla' \xi + \nabla' \bar v \otimes \nabla' v + \nabla' v \otimes \nabla'\bar v = 0 
	\end{equation}
	and
	\[
		\int_S f \cdot R \umatrix{\xi} \, dx' = 0 \,.
	\]
\end{lemma}

\begin{proof}
	Suppose $a(R) \neq 0$ and let $\lambda \in \r$ be such that $v(x') = -\lambda \dfrac{c(R)}{a(R)} x_1 + \lambda x_2$. For $\bar v \in W^{2, 2}(S)$ it is sufficient to define
	\[
		\xi(x') = \lambda \bar v(x') \begin{pmatrix}
			-\frac{c(R)}{a(R)} \\
			1
		\end{pmatrix} \,.
	\]
	Note that
	\[
		\umatrix{\xi} = \lambda W \begin{pmatrix}
			0 \\ 0 \\ \bar v
		\end{pmatrix} \quad \text{with} \quad W =  \begin{pmatrix}
		0 & 0 & -\frac{c(R)}{a(R)}\\
		0 & 0 & 1\\
		\frac{c(R)}{a(R)} & -1 & 0
		\end{pmatrix} \in \tang_R
	\]
	by \Cref{prop:condition_coefficients_abc}. In particular,
	\begin{equation}
		\int_S f \cdot R \umatrix{\xi} \, dx' = \lambda \int_S f \cdot RW \lmatrix{\bar v} \, dx' \,. \label{eq:equation_xi_v}
	\end{equation}
	Define the map $\Phi(t) = R e^{tW}$ for $t \in \r$. By \cite[Lemma 4.4]{MAOR2021}, $\Phi(t) \in \orot$ for any $t \in \r$, therefore
	\[
		\int_S f \cdot \Phi(t) \lmatrix{\bar v} \, dx' = 0 \quad \forall \, t \in \r\,,
	\]
	since $\Phi(t)^T f \cdot e_3 = 0$. Differentiating with respect to $t$ at $t = 0$, we deduce
	\[
		\int_S f \cdot RW \lmatrix{\bar v} \, dx' = 0 \,,
	\]
	which gives the thesis by \eqref{eq:equation_xi_v}.
\end{proof}

Having all the previous results at our disposal we can show that $J^{\VK}$ enjoys some invariance properties. 

\begin{prop}\label{prop:invariance_JVK}
	Suppose $R^T f \cdot e_3 = 0$ for every $R \in \orot$ and $\dim \orot = 1$. Let $\bar v \in \mathcal{V}_{R}$ and $\bar u \in \mathcal{U}_{R}$ be such that
	\[
		\nabla' \bar u^T + \nabla' \bar u + \nabla' \bar v \otimes \nabla' \bar v = 0 \,.
	\]
	Then $J^{\VK}( u + \bar u + \xi,  v +\bar v, R, W) = J^{\VK}(u, v, R, W)$ for every admissible quadruplet $(u, v, R, W)$, where $\xi$ is defined as in \Cref{lemma:definition_xi}.
\end{prop}

\begin{proof}
	Since $v$ is affine we immediately have that $(\nabla')^2(v+\bar v) = (\nabla')^2 v$. Moreover, by definition of $\xi$
	\begin{align}
		& (\nabla'(u + \bar u + \xi))^T + \nabla'(u + \bar u + \xi) + \nabla'(v+\bar v) \otimes \nabla' (v + \bar v)\\
		& \quad = (\nabla' u)^T + \nabla' u + \nabla' v \otimes \nabla' v \,.
	\end{align}
	By \Cref{lemma:affine_minimizers_zero_u,lemma:definition_xi}, to conclude we just need to show that
	\[
		\int_S f \cdot R W \lmatrix{\bar v} \, dx'  = 0 \,.
	\]
	This easily follows from the specific structure of $\bar v$. Indeed, suppose $a(R) \neq 0$ and let $\lambda \in \r$ be such that $\bar v(x') = -\lambda \dfrac{c(R)}{a(R)} x_1 + \lambda x_2$. Then
	\begin{align}
		\int_S f \cdot R W \lmatrix{\bar v} \, dx' & = \lambda\left(-W_{13}c(R) + W_{13}c(R) - W_{23} \frac{c^2(R)}{a(R)} + W_{23}b(R)\right) \\
		& = \lambda\left(- W_{23} \frac{c^2(R)}{a(R)} + W_{23}b(R)\right) = 0\,,
	\end{align}
	since $a(R)b(R) = c^2(R)$ by \Cref{prop:condition_coefficients_abc}.
\end{proof}

We are finally ready to give the proof of \Cref{teo:JVK_has_minimum}.

\begin{proof}[Proof of \Cref{teo:JVK_has_minimum}]
	Let $(u_n, v_n, R_n, W_n)$ be a minimizing sequence for $J^{\VK}$. Let $P_n^{\mathcal{V}}$ be the projection of $W^{2,2}(S)$ onto $\mathcal{V}_{R_n}$. By \Cref{prop:invariance_JVK} and the fact that
	\[
		J^{\VK}(u_n + Ax' + \eta, v_n + \delta, R_n, W_n) = J^{\VK}(u_n, v_n, R_n, W_n) \quad \forall A \in \r^{2 \times 2}_{\skw},  \eta \in \r^2,  \delta \in \r\,,
	\]
	we can suppose that for all $n \in \n$
	\begin{enumerate}[(i)]
		\item $\displaystyle \int_S u_n \, dx' = 0$ \label{item:projection_u}
		\item $\displaystyle \int_S v_n \, dx' = 0$ \label{item:mean_v}
		\item $P_n^{\mathcal{V}}(v_n) = 0$, \label{item:projection_v}
		\item $\displaystyle\int_S \skw(\nabla' u_n) \, dx' = 0$. \label{item:un_skew}
	\end{enumerate}
	Up to a subsequence, we can always assume that $R_n \to R \in \orot$. 
	
	Assume first that $\|u_n\|_{W^{1, 2}} + \|v_n\|^2_{W^{2, 2}} + |W_n|^2 \leq C$. Then, up to a subsequence we have $u_n \todeb u$ in $W^{1, 2}(S; \r^2)$, $v_n \todeb v$ in $W^{2, 2}(S)$ and $W_n \to W$ with $W \in \normal_R$. By lower semicontinuity of $J^{\VK}$ we deduce that $(u, v, R, W)$ is a minimizer of $J^{\VK}$ via the direct method of the Calculus of Variations. 
	
	Suppose now by contradiction that
	\[
		\|u_n\|_{W^{1, 2}} + \|v_n\|^2_{W^{2, 2}} + |W_n|^2 = \gamma_n^2 \to +\infty
	\] 
	and define $\bar u_n = \frac{1}{\gamma_n^2}u_n$, $\bar v_n = \frac{1}{\gamma_n} v_n$ and $\bar W_n = \frac{1}{\gamma_n} W_n$. Then, up to a subsequence, we have $\bar u_n \todeb \bar u$ in $W^{1, 2}(S; \r^2)$, $\bar v_n \todeb \bar v$ in $W^{2, 2}(S)$ and $\bar W_n \to \bar W$ with $\bar W \in \normal_R$. Since $J^{\VK}(u_n, v_n, R_n, W_n) \leq C$, we have
	\begin{equation}\label{eq:bound_gamma}
		\begin{aligned}
			C & \geq \gamma_n^4 \int_S \bar Q((\nabla'\bar u_n)^T + \nabla' \bar u_n + \nabla' \bar v_n \otimes \nabla' \bar v_n) \, dx' + \gamma_n^2 \int_S \bar Q((\nabla')^2 \bar v_n) \, dx'\\
			& \hphantom{\geq} \, \, -\gamma_n^2 \int_S f \cdot R_n \umatrix{\bar u_n} \, dx' - \gamma_n^2\int_S f \cdot R_n \bar W_n \lmatrix{\bar v_n} \, dx' \\
			& \hphantom{\geq} \, \, - \gamma_n^2 \int_S f \cdot R_n (\bar W_n)^2 \xprimo \, dx' \,.
		\end{aligned}
	\end{equation}
	Dividing by $\gamma_n^4$ we get by the coercivity of $\bar Q$
	\begin{equation}\label{eq:convergence_constraint}
		\|(\nabla' \bar u_n)^T + \nabla' \bar u_n + \nabla' \bar v_n \otimes \nabla'\bar v_n\|_{L^2} \leq \dfrac{C}{\gamma_n}
	\end{equation}
	and passing to the limit we deduce that $(\bar u, \bar v) \in \aisolin$.
    Moreover, dividing \eqref{eq:bound_gamma} by $\gamma_n^2$ and passing to the limit we get by lower semicontinuity that $0 \geq J^{\VK}(\bar u, \bar v, R, \bar W)$.
    The stability condition \ref{item:stability_VK} implies that $J^{\VK}(\bar u, \bar v, R, \bar W)$ is zero and $\bar v$ is affine.
    By \Cref{prop:structure_affine_minimizers} and the properties \ref{item:projection_u}--\ref{item:un_skew} we deduce that $\bar u = 0$, $\bar v = 0$ and $\bar W = 0$. 
    If we prove that $\bar u_n$ and $\bar v_n$ are strongly converging, then the proof is concluded since we would have
	\[
		\|\bar u\|_{W^{1, 2}} + \|\bar v\|^2_{W^{2, 2}} + |\bar W|^2 = 1 \,.
	\]
	
	Dividing \eqref{eq:bound_gamma} by $\gamma_n^2$ and passing to the limit we have
	\[
		0 \geq \limsup_{n \to \infty} \int_S \bar Q((\nabla')^2 \bar v_n) \, dx' \,.
	\]
	In particular, by the coercivity of $Q$ we get $(\nabla')^2 \bar v_n \to 0$ in $L^2(S; \r^{2 \times 2})$, giving the strong convergence of $\bar v_n$ in $W^{2, 2}(S)$. By \eqref{eq:convergence_constraint} we have that $\sym(\nabla' \bar u_n) \to 0$ in $L^2(S; \r^{2 \times 2})$. By \ref{item:un_skew} we can apply Korn inequality to deduce that $\bar u_n \to 0$ strongly in $W^{1, 2}(S; \r^2)$, concluding the proof of the first part. 
	
	Suppose now that \ref{item:stability_VK} fails. Let $(\bar v, \bar u) \in \aisolin$ such that for some $\bar R \in \orot$ and $\bar W \in \normal_{\bar R}$ either $J^{\VK}(\bar y, \bar v,\bar R, \bar W) < 0$ or $J^{\VK}(\bar y, \bar v, \bar R, \bar W) = 0$ and $\bar v$ is not affine. In any of these two cases, we have
	\begin{equation}
		- \int_S f \cdot \bar R \umatrix{\bar u} \, dx' - \int_S f \cdot \bar R\bar W \lmatrix{\bar v} \, dx' - \int_S f \cdot \bar R\bar W^2 \umatrix{x'}\, dx' < 0 \,.
	\end{equation}
	In particular we have that $J^{\VK}_{\varepsilon}(\bar u, \bar v, \bar R, \bar W) < 0$ for every choice of $\varepsilon > 0$. Since for every $\gamma > 0$ we have that $(\gamma^2 \bar u, \gamma \bar v) \in \aisolin$ and
	\begin{equation}
		\begin{aligned}
			& J^{\VK}_\varepsilon(\gamma^2 \bar u, \gamma \bar v, \bar R, \gamma \bar W) = \gamma^2\int_S \bar Q((\nabla')^2 v) \, dx' -\gamma^2(1+\varepsilon) \int_S f \cdot \bar R \umatrix{\bar u} \, dx' \\
			& \quad -\gamma^2(1+\varepsilon) \int_S f \cdot \bar R \bar W \lmatrix{\bar v} \, dx' - \gamma^2 (1+\varepsilon)\int_S f \cdot \bar R \bar W^2 \xprimo \, dx' \,,
		\end{aligned}
	\end{equation}
	we deduce that
	\[
		\lim_{\gamma \to +\infty} \dfrac{1}{\gamma^2} J^{\VK}_\varepsilon(\gamma^2 \bar u, \gamma \bar v, \bar R, \gamma \bar W) = J^{\VK}_\varepsilon (\bar u, \bar v, \bar R, \bar W) < 0 \,.
	\]
	This implies that 
	\[
		\lim_{\gamma \to +\infty} J^{\VK}_\varepsilon(\gamma^2 \bar u, \gamma \bar v, \bar R, \gamma \bar W) = - \infty \,,
	\]
	as desired.
\end{proof}

\begin{remark}\label{remark:eps0}
	From the proof it follows that $\inf J^{\VK} = -\infty$ if there is an admissible quadruplet $(\bar u, \bar v, \bar R, \bar W)$ such that $(\bar u, \bar v) \in \aisolin$ and $J^{\VK}(\bar u, \bar v, \bar R, \bar W) < 0$. In this case, one can repeat the same argument with $\varepsilon = 0$.
\end{remark}

\begin{remark}\label{remark:dim_R_0}
	We give a short sketch of the proof of \Cref{teo:JVK_has_minimum} in the case $\dim \orot = 0$. First of all, we can assume without loss of generality that $\orot = \{\Id\}$. Reasoning as in \Cref{prop:condition_coefficients_abc}, since $\normal_{\Id} = \r^{3 \times 3}_{\skw}$, one can show that $ab - c^2 > 0$, where we have written $a, b$, and $c$ in place of $a(\Id), b(\Id)$, and $c(\Id)$. Then, arguing as in \Cref{prop:structure_affine_minimizers}, one can prove that, when \ref{item:stability_VK} holds, any minimizer $(u, v, R, W)$ of $J^{\VK}$ with $(u, v) \in \aisolin$ is of the form $(\eta, \delta, \Id, 0)$, with $\eta \in \r^2$ and $\delta \in \r$. Note that, in this setting, stability condition \ref{item:stability_VK} basically reduces to the \textit{linearized stability} of \cite{LECUMBERRY2009} without imposing any additional Dirichlet condition on the boundary. Finally, one can argue as in the proof of \Cref{teo:JVK_has_minimum} to conclude.
\end{remark}

%% file: gamma_convergence.tex
\section{\texorpdfstring{$\bm{\Gamma}$}{\unichar{"0393}}-convergence of the elastic energy}\label{section:gamma_convergence}

The $\Gamma$-convergence of $\frac{1}{h^\alpha} E_h$ when $2 \leq \alpha \leq 4$ is due to Friesecke, James, and M\"uller in a series of works \cite{FRIESECKE2002, FRIESECKE2006}. For the convenience of the reader, we will state here their main results. All the proofs can be found in the aforementioned papers. The key ingredient is the well-known rigidity estimate proved by the same authors in \cite{FRIESECKE2002}.
To conclude the appendix, we prove the energy scaling of some test deformations, inspired by the recovery sequence constructions done in \cite{FRIESECKE2006}.

\begin{teo}[Rigidity estimate]\label{teo:rigidity_estimate}
	Let $(y_h) \subset W^{1,2}(\Omega; \r^3)$ and define
	\[
	D_h = \| \dist(\nabla_h y_h, \SO(3)) \|_{L^2(\Omega)}\,.
	\]
	There are two maps $R_h \in L^\infty(S;\SO(3))$ and $\tilde R_h \in W^{1,2}(S; \r^{3 \times 3}) \cap L^\infty(S; \r^{3 \times 3})$ such that
	\begin{enumerate}[start=1,label={(R\arabic*)}]
		\item $\|\nabla_h y_h - R_h\|_{L^2(\Omega)} \leq C D_h$ \label{item:approximated_rotation_convergence_scaled_gradient},
		\item $\|\nabla' \tilde R_h\|_{L^2(S)} \leq C h^{-1}D_h$\label{item:approximated_rotation_convergence_gradient},
		\item $\|\tilde R_h - R_h\|_{L^2(S)} \leq C D_h$,
		\item $\|\tilde R_h - R_h\|_{L^\infty(S)} \leq C h^{-1}D_h$.
	\end{enumerate}
	Moreover, there exist constant rotations $Q_h \in \SO(3)$ such that
	\begin{equation}
	\|R_h - Q_h\|_{L^2(S)} \leq Ch^{-1}D_h \,.\label{eq:approximated_constant_rotation}
	\end{equation}
	Finally, if $h^{-1}D_h \to 0$, then for $h \ll 1$ one can take $\tilde R_h = R_h$.
\end{teo}

First, we recall the compactness properties of sequences with bounded rescaled energy. We split the results into two cases, one for the Kirchhoff regime and one for the (constrained) Von K\'arm\'an regime.

\begin{prop}[Compactness in the Kirchhoff regime]\label{prop:compactness_alpha_2}
	Let $(y_h) \subset W^{1,2}(\Omega; \r^3)$ be a sequence such that $\frac{1}{h^2}E_h(y_h) \leq C$. Then there is $y \in \aiso$ such that, up to a subsequence, $y_h \to y$ in $W^{1, 2}(\Omega; \r^3)$ and
	\[
		\nabla_h y_h \to (\nabla' y, \nu) \quad \text{in} \, \, L^2(\Omega; \r^{3 \times 3}) \,,
	\]
	where $\nu= \partial_1 y \land \partial_2 y$.
\end{prop}

\begin{prop}[Compactness in the Von K\'arm\'an regime]\label{prop:compactness_alpha_greater_2}
	Let $(y_h) \subset W^{1,2}(\Omega; \r^3)$ be a sequence of deformations and let $(D_h) \subset \r^+$ be a sequence such that:
	\begin{enumerate}[(i)]
		\item $\frac{D_h}{h^2} \to 0$,
		\item $\limsup_{h \to 0} \frac{1}{D_h} E_h(y_h) \leq C$.
	\end{enumerate}
	Then there are constant rotations $R_h \in \SO(3)$ and constant vectors $c_h \in \r^3$ such that, setting
	\begin{equation}
	\tilde y_h = R_h^T y_h + c_h \,,
	\end{equation}
	we have
	\begin{enumerate}[(a)]
		\item $u_h \todeb u$ in $W^{1,2}(S; \r^2)$, \label{item:convergence_u}
		\item $v_h \to v$ in $W^{1,2}(S)$ with $v \in W^{2, 2}(S)$, \label{item:convergence_v}
	\end{enumerate}
	where
	\begin{align}
	& u_h \colon S \to \r^2 && x' \mapsto \min\left\{\dfrac{1}{\sqrt{D_h}}, \dfrac{h^2}{D_h}\right\}  \int_I \left(\tilde y_h' - x'\right) \, dx_3\,, \label{eq:definition_u}\\
	& v_h \colon S \to \r && x' \mapsto \dfrac{h}{\sqrt{D_h}} \int_I \tilde y_{h, 3} \, dx_3 \,. \label{eq:definition_v}
	\end{align}
	Lastly, if $\frac{D_h}{h^4} \to + \infty$, then $(u, v) \in \aisolin$.
\end{prop}

To conclude, we recall the $\Gamma$-convergence results.

\begin{teo}[$\Gamma$-convergence for the Kirchhoff regime]\label{teo:gamma_convergence_alpha_2}
	We have the following.
	\begin{enumerate}[(i)]
		\item For any sequence $(y_h) \subset W^{1,2}(\Omega; \r^3)$ such that $y_h \to y$ in $W^{1, 2}(\Omega; \r^3)$ for some $y \in \aiso$ it holds
		\[
		\liminf\limits_{h \to 0} \frac{1}{h^2}E_h(y_h) \geq E^{\KL}(y)\,.
		\]\label{item:liminf_inequality_alpha_2}
		\item For any $y \in \aiso$ there exists a sequence $(y_h) \subset W^{1,2}(\Omega; \r^3)$ such that $y_h \to y$ in $W^{1, 2}(\Omega; \r^3)$,
		\[
		\lim_{h \to 0} \frac{1}{h^2}E_h(y_h) = E^{\KL}(y) \,,
		\]\label{item:recovery sequence_alpha_2}
		and $\nabla_h y_h \to (\nabla' y, \nu)$ in $L^2(\Omega; \r^{3 \times 3})$.
	\end{enumerate}
\end{teo}

\begin{teo}[$\Gamma$-convergence for the Von K\'arm\'an regime]\label{teo:gamma_convergence_alpha_greater_2}
	Let $(D_h) \subset \r^+$ be a sequence such that $\frac{D_h}{h^2} \to 0$ and $D_h \geq Ch^4$ for $h \ll 1$. We have the following results.
	\begin{enumerate}[(i)]
		\item For any sequence $(y_h) \subset W^{1,2}(\Omega; \r^3)$ such that
		\begin{enumerate}[(a)]
			\item $u_h \todeb u$ in $W^{1,2}(S; \r^2)$,
			\item $v_h \to v$ in $W^{1,2}(S)$ with $v \in W^{2, 2}(S)$,
		\end{enumerate}
		where $u_h, v_h$ are defined as in \eqref{eq:definition_u}--\eqref{eq:definition_v} with $y_h$ in place of $\tilde y_h$ we have
		\begin{equation}
			\liminf\limits_{h \to 0} \frac{1}{D_h}E_h(y_h) \geq E^{\VK}(u, v) \,.
		\end{equation}\label{item:liminf_inequality}
		\item Assume that $\frac{D_h}{h^4} \to +\infty$. Then for any $(u, v) \in \aisolin$ there exists a sequence $(y_h) \subset W^{1,2}(\Omega; \r^3)$ such that
		\begin{enumerate}[(a)]
			\item $u_h \todeb u$ in $W^{1,2}(S; \r^2)$,
			\item $v_h \to v$ in $W^{1,2}(S)$ with $v \in W^{2, 2}(S)$,
		\end{enumerate}
		and
		\[
		\lim\limits_{h \to 0} \frac{1}{D_h}E_h(y_h) = E^{\VK}(u, v) \,,
		\]
		where $u_h$ and $v_h$ are defined as in \eqref{eq:definition_u}--\eqref{eq:definition_v} with $y_h$ in place of $\tilde y$.  \label{item:recovery_sequence_alpha_grater_2}
		\item Assume that $\frac{D_h}{h^4} \to 1$. Then for any $(u, v) \in W^{1, 2}(S; \r^2) \times W^{2, 2}(S)$ there exists a sequence $(y_h) \subset W^{1,2}(\Omega; \r^3)$ such that
		\begin{enumerate}[(a)]
			\item $u_h \todeb u$ in $W^{1,2}(S; \r^2)$,
			\item $v_h \to v$ in $W^{1,2}(S)$ with $v \in W^{2, 2}(S)$,
		\end{enumerate}
		and
		\[
		\lim\limits_{h \to 0} \frac{1}{D_h}E_h(y_h) = E^{\VK}(u, v) \,,
		\]
		where $u_h$ and $v_h$ are defined as in \eqref{eq:definition_u}--\eqref{eq:definition_v} with $y_h$ in place of $\tilde y$. \label{item:recovery_sequence_alpha_4}
	\end{enumerate}
\end{teo}

\begin{prop}\label{prop:fine_test_deformation}
    Let $R \in \SO(3)$ and $v \in C^\infty(\bar S)$.
    Consider the test deformation
	\begin{equation}
		\hat y_h(x', x_3) = R \begin{pmatrix}
			x' \\ hx_3
		\end{pmatrix} + R \begin{pmatrix}
			- h^2x_3 \nabla v^T \\ hv
		\end{pmatrix} \,.
	\end{equation}
    Then $ E_{h}(\hat y_{h}) = O(h^{4}) $.
\end{prop}

\begin{proof}
	We have
	\[
		\nabla_h \hat y_h = R\left(\Id + \begin{pmatrix}
			-h^2 x_3 \nabla^2 v & -h\nabla v^T \\
			h \nabla v & 0
		\end{pmatrix}\right) \,.
	\]
	Hence,
	\[
		\nabla_h \hat y_h^T \nabla_h \hat y_h = \Id - 2h^2 x_3\begin{pmatrix}
		\nabla^2 v & 0 \\
		0 & 0
		\end{pmatrix} -h^2\nabla v \otimes \nabla v +O(h^3) \,.
	\]
	Expanding the square root around the identity we get
	\[
		\sqrt{\nabla_h \hat y_h^T \nabla_h \hat y_h} = \Id -h^2 x_3 \begin{pmatrix}
		\nabla^2 v & 0 \\
		0 & 0
		\end{pmatrix} -\frac{h^2}{2}\nabla v \otimes \nabla v +O(h^3) \,.
	\]
	Finally, since by frame-indifference $W(\nabla_h \hat y_h) = W\left(\sqrt{\nabla_h \hat y_h^T \nabla_h \hat y_h}\right)$, expanding the energy $W$ we have
	\[
		W(\nabla_h \hat y_h) = \frac{1}{2}h^4Q\left(x_3\begin{pmatrix}
		\nabla^2 v & 0 \\
		0 & 0
		\end{pmatrix} + \frac{1}{2}\nabla v \otimes \nabla v\right) + o(h^4) \,.
	\]
    Integrating over $ S $ and recalling that $ o(h^{4}) $ is uniform in $ S $, we conclude.
\end{proof}

\begin{prop}\label{prop:fine_test_deformation_det}
    Let $ (D_{h}) \subset \mathbb{R}^{+} $ be an infinitesimal sequence such that $ \frac{D_{h}}{h^{2}} \to 0 $.
    Let $v \in C^\infty(\bar S)$ such that $\det((\nabla')^2 v) = 0$ in $S$.
    Let $\tilde u_h \in W^{2, \infty}(S; \r^2)$ be such that
	\[
		\tilde y_h(x') = \xprimo + \begin{pmatrix}
			h^{-2}D_h \tilde u_h \\
			h^{-1}\sqrt{D_h}v
		\end{pmatrix} \,
	\]
	is an isometric embedding, i.e., $ \nabla' \tilde y_{h}^{T} \nabla' \tilde y_{h} = \Id $.
    Moreover, suppose that $ \|u_{h}\|_{W^{2, \infty}} \leq C $.
    Let $ R \in \SO(3) $ and define 
	\begin{equation}
		\hat y_h = R\tilde y_h + h x_3 R\nu_h \,,
	\end{equation}
	where $\nu_h  = \partial_1 \tilde y_h \land \partial_2 \tilde y_h$.
    Then $ E_{h}(\hat y_{h}) = O(D_{h}) $.
\end{prop}

\begin{proof}
    Setting
	\[
		\hat R_h = \begin{pmatrix}
		\nabla' \tilde y_h & \nu_h
		\end{pmatrix} \,,
	\]
	we have
	\[
		\nabla_h \hat y_h = R\hat R_h(\Id + hx_3\hat R_h^T (\nabla' \nu_h \, \,  0)) \,.
	\]
	It is easily found that
	\[
		\nabla' \nu_h = -h^{-1} \sqrt{D_h} \begin{pmatrix}
                (\nabla')^2 v  \\
                0
            \end{pmatrix} + O(h^{- 2}D_h) \,,
	\]
	thus
	\[
		\nabla_h \hat y_h = R\hat R_h\left(\Id + \sqrt{D_h} x_3\hat R_h^T \begin{pmatrix}
                (\nabla')^{2} v & 0\\
                0 & 0
		\end{pmatrix} + O(h^{-1}D_h)\right) \,.
	\]
	Expanding $W$ near the identity, one gets by frame-indifference that $E_h(\hat y_h) = O(D_h)$.
\end{proof}

%% file: properties_R.tex
\section{Fine properties of optimal rotations}\label{app:optimal_rotations}

In this appendix, we recall some properties of optimal rotation,  and we further analyse their structure in our specific setting. We will restrict ourselves to dead loads of body type. However, the same analysis can be carried out for surface dead loads. Consider a non-zero force $f \in L^2(S; \r^3)$ and suppose that
\[
	\int_S f \, dx' = 0 \,.
\]
We define the set of optimal rotations as

\[
	\orot = \argmax_{R \in SO(3)} F(R) \,,
\]
where
\[
F(A) = \int_S f \cdot A \xprimo \, dx' = \int_\Omega f \cdot A \begin{pmatrix}
x' \\ x_3
\end{pmatrix}\, dx \,.
\]
For the convenience of the reader, we recall here some properties of $\orot$ proved in \cite{MAOR2021}. The set $\orot$ is a closed, connected, boundaryless, and totally geodesic submanifold of $SO(3)$ (see \cite[Proposition 4.1]{MAOR2021}). 

The set of rotations can be equipped with its intrinsic distance
\begin{equation}
	\dist_{\SO(3)}(Q, R) = \min \left\{|W| \colon W \in \rskew \,,\,  Q = Re^W \right\} \,, \label{eq:definition_intrinsic_distance}
\end{equation}
for every $Q, R \in \SO(3)$. The tangent space to $\orot$ at the point $R$ is denoted with $\tang_R$ and is given by
\begin{equation}
	\tang_R = \left\{W \in \rskew \colon F(RW^2) = 0\right\} \label{eq:definition_tangent}\,.
\end{equation}
Note that by differentiating the map $t \mapsto F(Re^{tW})$ and evaluating it at $t=0$ we obtain
\begin{equation}
	F(RW) = 0 \,,\, F(RW^2) \leq 0 \quad \forall \, R \in \orot \quad \forall \, W \in \rskew \,. \label{eq:condition_maximality}
\end{equation}
The normal space to $\orot$ at the point $R$ is denoted by $\normal_R$ and is given by
\begin{equation} \label{eq:definition_normal}
	\normal_R = \left\{W \in \rskew \colon W : W' = 0 \quad \forall \, W' \in \tang_R\right\}\,.
\end{equation}
Observe that for any non-zero skew-symmetric matrix $W$ in $\normal_R$ one has $F(RW^2) < 0$. 

We start the section giving an example of force for which the compatibility condition \eqref{eq:compatibility} is satisfied. 

\begin{example}
	Consider $S = (-\frac{1}{2}, \frac{1}{2})^2$ and $f = x_1e_3$. A quick computation gives
	\[
	F(R) = \dfrac{1}{12}R_{31} \,,
	\]
	thus, $\orot = \{R \in \SO(3) \colon R_{31} = 1\}$. In particular for any optimal rotation $R \in \orot$ we have $R^T e_3 = e_1$, so that $R^T f \cdot e_3 = (x_1 - \frac{1}{2})e_1 \cdot e_3 = 0$.
\end{example}

We start now to prove some further properties of optimal rotations, valid for our setting. 

\begin{lemma}\label{lemma:dimension_R}
	The dimension of $ \orot $ is not $ 2 $.
\end{lemma}

\begin{proof}
	Let $\bar R \in \orot$.
    Define
	\[
		\tilde F(A) = \int_\Omega f \cdot \bar R A \begin{pmatrix}
			x' \\ 0
		\end{pmatrix} \, dx = \int_\Omega \bar R^T f \cdot A \begin{pmatrix}
		x' \\0 
		\end{pmatrix} \, dx\,.
	\]
	Similarly, define
	\[
		\tilde \orot= \argmax_{R \in \SO(3)} \tilde F(R) \,.
	\]
	Note that $\orot = \bar R \cdot \tilde \orot$ so it is enough to prove that $\dim \tilde \orot \neq 2$.  Clearly $\Id \in \tilde \orot$, so we can use the classification of \cite[Proposition 6.2]{MAOR2021}.
    Since $\tilde F$ and $F$ are linear on the space of $3 \times 3$ matrices, we can represent them by $3 \times 3$ matrices, that we will still denote, with a slight abuse of notation, by $\tilde F$ and $F$.
    By \cite[Proposition 6.2]{MAOR2021}, $\orot$ is two-dimensional when the eigenvalues of $\tilde F$ are of the form $a, a, -a$ for some $a > 0$. Note first of all that
	\[
		\tilde F \colon A = F:\bar R A \quad \forall \, A \in \r^{3 \times 3} \,,
	\]
	so that $\tilde F = \bar R^T F$. Moreover,
	\[
		F_{i3} = F:E^{i3} = F(E^{i3}) = 0 \quad i = 1, 2, 3 \,,
	\]
	where $E^{ij}$ is the matrix such that $E_{km}^{ij} = \delta_{ki}\delta_{mj}$ and $\delta_{ij}$ is the usual Kronecker symbol.
    It follows that $\det(\tilde F) = \det(F) = 0$ and $0$ is an eigenvalue of $\tilde F$, concluding the proof.
\end{proof}

\begin{remark}\label{remark:dimension_R}
    If \eqref{eq:compatibility} is in force, then we also have $ \dim \orot \neq 3 $, thus $ \dim \orot $ is either a singleton or a closed geodesic (see \cite[Proposition 6.2]{MAOR2021}).
    However, as showed in \cite{MAOR2021}, we can have non-zero forces for which $ \orot = \SO(3) $.
    As an example consider $ f = (1 - \frac{3}{2}|x|)e_{1} $ acting on $ S = B_{1} $.
    Then
    \begin{equation}
        F(R) = \int_{B_{1}} f(x) \cdot R \xprimo \, dx' = \int_{0}^{1} \int_{0}^{2\pi} r\left(1 - \frac{3}{2}r\right)e_{1} \cdot R \begin{pmatrix}
            r\cos\theta\\
            r\sin\theta\\
            0
        \end{pmatrix} \, d\theta \, dr = 0. 
    \end{equation}
    In particular $ \orot = \SO(3) $.
    In this case \eqref{eq:compatibility} does not hold.
\end{remark}

The set of rotations is not linear. However, the 2-dimensional structure of the integral that defines $F$ gives us the freedom to perform some change of sign to the columns of a rotation while keeping the sign of its determinant. A few simple results follow from this observation.

\begin{lemma}\label{lemma:sign_F}
	If \eqref{eq:compatibility} holds, then $\max\limits_{R \in \SO(3)} F(R) > 0$. Otherwise, $\max\limits_{R \in \SO(3)} F(R) \geq 0$.
\end{lemma}

\begin{proof}
	Assume \eqref{eq:compatibility} and suppose by contradiction that $F(R) \leq 0$ for any rotation $R \in \SO(3)$. By \eqref{eq:compatibility} we have $\orot \neq \SO(3)$, hence the map $F$ can not vanish on the whole $\SO(3)$. Thus, there is a rotation $R$ such that $F(R) < 0$. Now consider the matrix
	\[
		\hat R = \begin{pmatrix}
			-R^1 & -R^2 & R^3
		\end{pmatrix} \,.
	\]	
	Note that $\hat R \in \SO(3)$ and $F(\hat R) = -F(R) > 0$. This gives the desired contradiction. The same argument applies to the second part of the statement.
\end{proof}

Consider now $R \in \orot$ and the skew-symmetric matrix
\[
	W = \begin{pmatrix}
		0 & 1 & 0\\
		-1 & 0 & 0\\
		0 & 0 & 0
	\end{pmatrix} \,.
\]
Since $F(RW) = 0$, we get
\begin{equation}\label{eq:c_symmetry}
	\int_S f \cdot R \begin{pmatrix}
		x_2 \\ 0\\ 0
	\end{pmatrix} \, dx' = \int_S f \cdot R \begin{pmatrix}
	0 \\ x_1\\ 0
	\end{pmatrix} \, dx'\,.
\end{equation}
For a given $R \in \orot$ we then define
\begin{align+}
	a(R) & = \int_S f \cdot R \begin{pmatrix}
	x_1 \\ 0\\ 0
	\end{pmatrix} \, dx' \,, \label{eq:definition_a} \\
	b(R) & = \int_S f \cdot R \begin{pmatrix}
	0 \\ x_2\\ 0
	\end{pmatrix} \, dx' \,, \label{eq:definition_b}\\
	c(R) & = \int_S f \cdot R \begin{pmatrix}
	0 \\ x_1\\ 0
	\end{pmatrix} \, dx' = \int_S f \cdot R \begin{pmatrix}
	x_2 \\ 0\\ 0
	\end{pmatrix} \, dx\,.\label{eq:definition_c}
\end{align+}
Note that by \Cref{lemma:sign_F} we have that $a(R) + b(R) = F(R) \geq 0$. Moreover, $a(R)$ and $b(R)$ can not be negative, as proved in the following lemma. In particular, when \eqref{eq:compatibility} holds, $a(R)$ and $b(R)$ cannot be both zero by \Cref{lemma:sign_F}.

\begin{lemma}\label{lemma:a_b_positive}
	It holds that $a(R), b(R) \geq 0$ for any $R \in \orot$.
\end{lemma}

\begin{proof}
	Suppose by contradiction that $a(R) < 0$ for some $R \in \orot$. Then by \Cref{lemma:sign_F} we have $b(R) = F(R) - a(R) \geq 0$. Consider the rotation
	\[
		\hat R = \begin{pmatrix}
			-R^1 & R^2 & -R^3
		\end{pmatrix} \in \SO(3)\,.
	\]	
	Then $F(R) \geq F(\hat R) = - a(R) + b(R) > a(R) + b(R) = F(R)$, which gives a contradiction. A similar proof can be done for $b(R)$.
\end{proof}

We can now give an explicit characterization of the tangent space $\tang_R$ in terms of the quantities $a(R), b(R)$ and $c(R)$.

\begin{prop}\label{prop:condition_coefficients_abc}
	Assume \eqref{eq:compatibility} and suppose that $\dim \orot=1$. Let $R \in \orot$. Then
	\[
		a(R)b(R) - c^2(R) = 0\,.
	\]
	Moreover,
	\begin{align}
		\tang_R & = \left\{W \in \rskew \colon W_{12} = 0\,, W_{13} = -\frac{c(R)}{a(R)}W_{23}\right\} && \text{if} \, \, a(R) \neq 0 \,,\\
		\tang_R & = \left\{W \in \rskew \colon W_{12} = 0\,, W_{23} = -\frac{c(R)}{b(R)}W_{13}\right\} && \text{if} \, \, b(R) \neq 0 \,.
	\end{align}
\end{prop}

\begin{proof}
	By \eqref{eq:definition_tangent} the tangent space to $\orot$ at $R$ is the set of zeros of the map $W \in \rskew \mapsto F(RW^2)$. For a general skew-symmetric matrix $W$, we have
	\[
		(W^2)' = -\begin{pmatrix}
			W_{12}^2 + W_{13}^2 & W_{13}W_{23} \\
			W_{13}W_{23} & W_{12}^2 + W_{23}^2 \\
		\end{pmatrix} \,.
	\]
	Hence, by \eqref{eq:compatibility} we have
	\begin{align}
		F(RW^2) & = -(W_{12}^2 + W_{13}^2)a(R) - 2W_{13}W_{23}c(R) - (W_{12}^2 + W_{23}^2)b(R)
	\end{align}
	This expression can be considered as a quadratic form $q: \r^3 \to \r$ computed at the vector $(W_{12}, W_{13}, W_{23})$. We can identify $q$ with a symmetric matrix and study its sign. We have
	\[
		q = - \begin{pmatrix}
			a(R) + b(R) & 0 & 0\\
			0 & a(R) & c(R)\\
			0 & c(R) & b(R)
		\end{pmatrix} \,,
	\]
	so by \Cref{lemma:sign_F}--\ref{lemma:a_b_positive} the sign of $q$ depends solely on the minor $a(R)b(R)-c^2(R)$. If $a(R)b(R)-c^2(R) > 0$, the only zero of $q$ is at $0$, contradicting the hypothesis on the dimension of $\orot$. If $a(R)b(R)-c^2(R) < 0$, the set of zeros of $q$ contains two lines that span a subset of dimension 2 in $\r^3$, contradicting again the assumption $\dim \orot= 1$. Therefore, it must hold that $a(R)b(R)-c^2(R) = 0$. In this case, we have
	\begin{align}
		q(W) & = -W_{12}^2F(R) - \left(W_{13} \sqrt{a(R)} + W_{23} \frac{c(R)}{\sqrt{a(R)}}\right)^2 && \text{if} \, \, a(R) \neq 0\,,\\
		q(W) & = - W_{12}^2F(R) - \left(W_{23} \sqrt{b(R)} + W_{13} \frac{c(R)}{\sqrt{b(R)}}\right)^2 && \text{if} \, \, b(R) \neq 0 \,,
	\end{align}
	concluding the characterization of the tangent space by \Cref{lemma:sign_F} ($F(R) > 0$).
\end{proof}

\begin{cor}\label{cor:structure_normal}
	Assume \eqref{eq:compatibility} and suppose that $\dim \orot= 1$ and let $R \in \orot$. Then
	\begin{align}
	\normal_R & = \left\{W \in \rskew \colon W_{23} = \frac{c(R)}{a(R)}W_{13}\right\} && \text{if} \, \, a(R) \neq 0 \,,\\
	\normal_R & = \left\{W \in \rskew \colon W_{13} = \frac{c(R)}{b(R)}W_{23}\right\} && \text{if} \, \, b(R) \neq 0 \,.
	\end{align}
\end{cor}

%% file: acknowledgments.tex
\noindent \textbf{Acknowledgments.} The author acknowledges support by PRIN2022 number 2022J4FYNJ funded by MUR, Italy, and by the European Union--Next Generation EU. 
The author acknowledges support by INdAM--GNAMPA Project CUP E5324001950001.

%% file: main.bbl
\begin{thebibliography}{10}

\bibitem{ABELS2010}
Helmut Abels, Maria~Giovanna Mora, and Stefan M\"{u}ller.
\newblock The time-dependent Von Kármán plate equation as a limit of 3d nonlinear elasticity.
\newblock {\em Calculus of Variations and Partial Differential Equations}, 41(1–2):241–259, August 2010.

\bibitem{ABELS2011}
Helmut Abels, Maria~Giovanna Mora, and Stefan M\"{u}ller.
\newblock Thin vibrating plates: long time existence and convergence to the Von Kármán plate equations.
\newblock {\em GAMM-Mitteilungen}, 34(1):97–101, April 2011.

\bibitem{CARROLL1989}
Sean~Michael Carroll and Bradley~William Dickinson.
\newblock Construction of neural nets using the radon transform.
\newblock {\em International 1989 Joint Conference on Neural Networks}, pages 607--611 vol.1, 1989.

\bibitem{CYBENKO1989}
George Cybenko.
\newblock Approximation by superpositions of a sigmoidal function.
\newblock {\em Mathematics of Control, Signals, and Systems}, 2(4):303–314, 1989.

\bibitem{FOPPL1907}
August F\"oppl.
\newblock {\em Vorlesung \"uber technische Mechanik}, volume~5.
\newblock Leipzig, 1907.

\bibitem{FRIEDRICH2020}
Manuel Friedrich and Martin Kružík.
\newblock Derivation of Von Kármán Plate Theory in the Framework of Three-Dimensional Viscoelasticity.
\newblock {\em Archive for Rational Mechanics and Analysis}, 238(1):489–540, June 2020.

\bibitem{FRIESECKE2002}
Gero Friesecke, Richard~D. James, and Stefan M{\"u}ller.
\newblock A theorem on geometric rigidity and the derivation of nonlinear plate theory from three-dimensional elasticity.
\newblock {\em Communications on Pure and Applied Mathematics}, 55(11):1461--1506, 2002.

\bibitem{FRIESECKE2006}
Gero Friesecke, Richard~D. James, and Stefan M{\"u}ller.
\newblock A hierarchy of plate models derived from nonlinear elasticity by $\Gamma$-convergence.
\newblock {\em Archive for rational mechanics and analysis}, 180:183--236, 2006.

\bibitem{HORNUNG2011}
Peter Hornung.
\newblock Approximation of Flat $W^{2, 2}$ Isometric Immersions by Smooth Ones.
\newblock {\em Archive for Rational Mechanics and Analysis}, 199(3):1015--1067, 2011.

\bibitem{VONKARMAN1910}
Theodore~Von K\'arm\'an.
\newblock {\em Festigkeitsprobleme im Maschinenbau}, volume IV/4.
\newblock Leipzig, 1910.

\bibitem{KIRCHOFF1850}
Gustav Kirchhoff.
\newblock \"{U}ber das Gleichgewicht und die Bewegung einer elastischen Scheibe.
\newblock {\em Journal f\"{u}r die reine und angewandte Mathematik (Crelles Journal)}, 1850(40):51--88, 1850.

\bibitem{LECUMBERRY2009}
Myriam Lecumberry and Stefan M\"{u}ller.
\newblock Stability of Slender Bodies under Compression and Validity of the Von K{\'{a}}rm{\'{a}}n Theory.
\newblock {\em Archive for Rational Mechanics and Analysis}, 193(2):255--310, 2009.

\bibitem{MADDALENA2018}
Francesco Maddalena, Danilo Percivale, and Franco Tomarelli.
\newblock Variational problems for {F}\"oppl--Von {K}\'arm\'an plates.
\newblock {\em SIAM J. Math. Anal.}, 50(1):251--282, 2018.

\bibitem{MAOR2021}
Cy~Maor and Maria~Giovanna Mora.
\newblock Reference Configurations Versus~Optimal Rotations: A Derivation of Linear Elasticity from Finite Elasticity for all Traction Forces.
\newblock {\em Journal of Nonlinear Science}, 31(3), 2021.

\bibitem{NEUKAMM2013}
Stefan Neukamm and Igor Velčić.
\newblock Derivation of a homogenized Von K\'arm\'an plate theory from 3D nonlinear elasticity.
\newblock {\em Mathematical Models and Methods in Applied Sciences}, 23(14):2701–2748, October 2013.

\bibitem{VELCIC2016}
Igor Velčić.
\newblock On the general homogenization of Von Kármán plate equations from three-dimensional nonlinear elasticity.
\newblock {\em Analysis and Applications}, 15(01):1–49, November 2016.

\end{thebibliography}
